\documentclass[11pt]{article}

\usepackage{xcolor}
\usepackage{amsmath}
\usepackage{amsfonts}
\usepackage{amssymb}
\usepackage{subfigure}
\usepackage{epsf,fancybox}
\usepackage{mathrsfs}
\usepackage{color}
\usepackage{multirow}
\usepackage{paralist}
\usepackage{verbatim}
\usepackage{galois}
\usepackage{algorithm}
\usepackage{algorithmic}
\usepackage{boxedminipage}
\usepackage{stmaryrd}
\definecolor{light-gray}{gray}{0.85}
\usepackage{colortbl}
\usepackage{booktabs}
\usepackage{makecell}
\usepackage{hhline}

\usepackage{natbib}

\usepackage{url}
\usepackage[colorlinks,linkcolor=magenta,citecolor=blue, pagebackref=true,backref=true]{hyperref}
\renewcommand*{\backrefalt}[4]{%
    \ifcase #1 \footnotesize{(Not cited.)}%
    \or        \footnotesize{(Cited on page~#2.)}%
    \else      \footnotesize{(Cited on pages~#2.)}%
    \fi}

\newtheorem{theorem}{Theorem}[section]
\newtheorem{corollary}[theorem]{Corollary}
\newtheorem{lemma}[theorem]{Lemma}

\newtheorem{definition}[theorem]{Definition}

\newcommand{\EE}{\mathbb{E}}

\newcommand{\grad}{\nabla}
\newcommand{\gradx}{\nabla_{\mathbf x}}
\newcommand{\grady}{\nabla_{\mathbf y}}
\newcommand{\kappax}{\kappa_{\mathbf x}}
\newcommand{\kappay}{\kappa_{\mathbf y}}
\newcommand{\Dx}{D_{\mathbf x}}
\newcommand{\Dy}{D_{\mathbf y}}

\newcommand{\st}{\textnormal{s.t.}}

\newcommand{\proj}{\mathcal{P}}
\newcommand{\PX}{\mathcal{P}_\mathcal{X}}
\newcommand{\PY}{\mathcal{P}_\mathcal{Y}}

\newcommand{\x}{\mathbf x}
\newcommand{\y}{\mathbf y}

\newcommand{\z}{\mathbf z}
\newcommand{\w}{\mathbf w}
\newcommand{\su}{\mathbf u}
\newcommand{\sv}{\mathbf v}
\newcommand{\argmin}{\mathop{\rm{argmin}}}
\newcommand{\argmax}{\mathop{\rm{argmax}}}

\newcommand{\XCal}{\mathcal{X}}
\newcommand{\YCal}{\mathcal{Y}}

\newcommand{\prox}{\textnormal{prox}}

\newcommand{\br}{\mathbb{R}}

\newcommand{\ba}{\begin{array}}
\newcommand{\ea}{\end{array}}

\newcommand{\norm}[1]{\|{#1} \|}

\newcommand{\cO}{\mathcal{O}}

\newcommand{\AGPROX}{Maximin-AG2}
\newcommand{\AGCC}{Minimax-APPA}
\newcommand{\AGNC}{Minimax-PPA}

\title{\bf{\LARGE{Near-Optimal Algorithms for Minimax Optimization}}}

\author{Tianyi Lin\thanks{Department of IEOR, UC Berkeley, Berkeley, CA 94720, USA; Email: darren{\_}lin@berkeley.edu.} \and Chi Jin\thanks{Department of EE, Princeton University, Princeton, NJ 08544, USA; Email: chij@princeton.edu.} \and Michael. I. Jordan\thanks{Department of EECS and Statistics, UC Berkeley, Berkeley, CA 94720, USA; Email: jordan@cs.berkeley.edu.}
}

\begin{document}
\maketitle

\begin{abstract}
This paper resolves a longstanding open question pertaining to the design of near-optimal first-order algorithms for smooth and strongly-convex-strongly-concave minimax problems. Current state-of-the-art first-order algorithms find an approximate Nash equilibrium using $\tilde{O}(\kappax+\kappay)$~\citep{Tseng-1995-Linear} or $\tilde{O}(\min\{\kappax\sqrt{\kappay}, \sqrt{\kappax}\kappay\})$~\citep{Alkousa-2019-Accelerated} gradient evaluations, where $\kappax$ and $\kappay$ are the condition numbers for the strong-convexity and strong-concavity assumptions. A gap still remains between these results and the best existing lower bound $\tilde{\Omega}(\sqrt{\kappax\kappay})$~\citep{Ibrahim-2019-Lower, Zhang-2019-Lower}.  This paper presents the first algorithm with $\tilde{O}(\sqrt{\kappax\kappay})$ gradient complexity, matching the lower bound up to logarithmic factors. Our algorithm is designed based on an accelerated proximal point method and an accelerated solver for minimax proximal steps.  It can be easily extended to the settings of strongly-convex-concave, convex-concave, nonconvex-strongly-concave, and nonconvex-concave functions. This paper also presents algorithms that match or outperform all existing methods in these settings in terms of gradient complexity, up to logarithmic factors.
\end{abstract}

\section{Introduction}
Let $\br^m$ and $\br^n$ be finite-dimensional Euclidean spaces and let the function $f: \br^m \times \br^n \rightarrow \br$ be smooth. Let $\XCal$ and $\YCal$ are two nonempty closed convex sets in $\br^m$ and $\br^n$. Our problem of interest is the following minimax optimization problem: 
\begin{equation}\label{prob:minimax}
\min_{\x \in \XCal} \max_{\y \in \YCal} \ f(\x, \y). 
\end{equation}
The theoretical study of solutions of problem~\eqref{prob:minimax} has been an  focus of several decades of research in mathematics, statistics, economics and computer science~\citep{Basar-1999-Dynamic, Nisan-2007-Algorithmic, Von-2007-Theory, Facchinei-2007-Finite, Berger-2013-Statistical}. Recently, this line of research has become increasingly relevant to algorithmic machine learning, with applications including robustness in adversarial learning~\citep{Goodfellow-2014-Generative, Sinha-2018-Certifiable}, prediction and regression problems~\citep{Cesa-2006-Prediction, Xu-2009-Robustness} and distributed computing~\citep{Shamma-2008-Cooperative, Mateos-2010-Distributed}. Moreover, real-world machine-learning systems are increasingly embedded in multi-agent systems or matching markets and subject to game-theoretic constraints~\citep{Jordan-2018-Artificial}.

 Most existing work on minimax optimization focuses on the convex-concave setting, where the function $f(\cdot, \y)$ is convex for each $\y \in \br^n$ and the function $f(\x, \cdot)$ is concave for each $\x \in \br^m$. The best known convergence rate in a general convex-concave setting is $O(1/\epsilon)$ in terms of duality gap, which can be achieved by Nemirovski's mirror-prox algorithm~\citep{Nemirovski-2004-Prox} (a special case of which is the extragradient algorithm~\citep{Korpelevich-1976-Extragradient}), Nesterov's dual extrapolation algorithm~\citep{Nesterov-2007-Dual} or Tseng's accelerated proximal gradient algorithm~\citep{Tseng-2008-Accelerated}. This rate is known to be optimal for the class of smooth convex-concave problems~\citep{Ouyang-2019-Lower}. Furthermore, optimal algorithms are known for special instances of convex-concave setting; e.g., for the affinely constrained smooth convex problem~\citep{Ouyang-2015-Accelerated} and problems with a composite bilinear objective function, $f(\x, \y) = g(\x) + \x^\top A\y - h(\y)$~\citep{Chen-2014-Optimal}. 

Very recently, the lower complexity bound of first-order algorithms have been established for solving general strongly-convex-strongly-concave and strongly-convex-concave minimax optimization problems~\citep{Ouyang-2019-Lower, Ibrahim-2019-Lower, Zhang-2019-Lower}. For the strongly-convex-strongly-concave setting, in which $\kappax, \kappay > 0$ are the condition numbers for $f(\cdot, \y)$ and $f(\x, \cdot)$, respectively, the complexity bound is $\tilde{\Omega}(\sqrt{\kappax\kappay})$ while the best known upper  bounds are $\tilde{O}(\kappax+ \kappay)$~\citep{Tseng-1995-Linear, Gidel-2018-A, Mokhtari-2019-Unified} and $\tilde{O}(\min\{\kappax\sqrt{\kappay}, \kappay\sqrt{\kappax}\})$ \citep{Alkousa-2019-Accelerated}. For the strongly-convex-concave setting in which $\kappax > 0$ and $\kappay = 0$, the lower complexity bound is $\tilde{\Omega}(\sqrt{\kappax/\epsilon})$ while the best known upper bound is $O(\kappax/\epsilon)$  \citep{Thekumparampil-2019-Efficient}. The existing algorithms that obtain a rate of $O(\sqrt{\kappax/\epsilon})$ in this context are only for special case of strongly-convex-linear, where $\x$ and $\y$ are connected only through a bilinear term $\x^\top A \y$ or $f(\x, \cdot)$ is linear for each $\x \in \br^m$ \citep[see, e.g.,][]{Nesterov-2005-Smooth, Chambolle-2016-Ergodic, Juditsky-2011-First, Hamedani-2018-Primal}. Thus, a gap remains between the lower complexity bound and the upper complexity bound for existing algorithms in both the strongly-convex-strongly-concave setting and the strongly-convex-concave setting. Accordingly, we have the following open problem:
\begin{center}
\textbf{Can we design first-order algorithms that achieve the lower bounds in these settings?}
\end{center}
This paper presents an affirmative answer by resolving the above open problem up to logarithmic factors. More specifically, our contribution is as follows.  
\begin{table*}[!t]
\renewcommand{\arraystretch}{1.4}
\centering
\caption{Comparison of gradient complexities to find an $\epsilon$-saddle point (Definition \ref{Def:eps_saddle}) in the convex-concave setting. This table highlights only the dependency on error tolerance $\epsilon$ and the strong-convexity and strong-concavity condition numbers, $\kappax, \kappay$.}
\label{tab:results_cc}
\begin{tabular}{|c|c|c|} \hline
 \textbf{Settings} & \textbf{References} & \textbf{Gradient Complexity} \\ \hline 
\multirow{8}{*}{\shortstack{Strongly-Convex-Strongly-\\Concave}} & \citet{Tseng-1995-Linear} & \multirow{4}{*}{$\tilde{O}(\kappax+\kappay)$} \\ \cline{2-2}
& \citet{Nesterov-2006-Solving} & \\ \cline{2-2}
& \citet{Gidel-2018-A} & \\ \cline{2-2}
& \citet{Mokhtari-2019-Unified} & \\ \cline{2-3} 
& \citet{Alkousa-2019-Accelerated} & $\tilde{O}(\min\{\kappax\sqrt{\kappay}, \kappay\sqrt{\kappax}\})$ \\\hhline{|~--|}
& \cellcolor{light-gray} \textbf{This paper} (Theorem \ref{thm:scsc}) & \cellcolor{light-gray} $\tilde{O}(\sqrt{\kappax\kappay})$ \\ \cline{2-3} 
& Lower bound \citep{Ibrahim-2019-Lower} & $\tilde{\Omega}(\sqrt{\kappax\kappay})$ \\ \cline{2-3} 
& Lower bound \citep{Zhang-2019-Lower} & $\tilde{\Omega}(\sqrt{\kappax\kappay})$ \\ \hline \hline
\multirow{4}{*}{\shortstack{Strongly-Convex-Linear\\{\footnotesize(special case of}\\{\footnotesize strongly-convex-concave)}}} & \citet{Juditsky-2011-First} & \multirow{3}{*}{$O(\sqrt{\kappax/\epsilon})$} \\ \cline{2-2}
& \citet{Hamedani-2018-Primal} & \\ \cline{2-2}
& \citet{Zhao-2019-Optimal} & \\ \cline{1-3}
\multirow{2}{*}{Strongly-Convex-Concave} & \citet{Thekumparampil-2019-Efficient} & $\tilde{O}(\kappax/\sqrt{\epsilon})$ \\ 
\hhline{|~--|}
& \cellcolor{light-gray} \textbf{This paper} (Corollary \ref{cor:scc}) & \cellcolor{light-gray} $\tilde{O}(\sqrt{\kappax/\epsilon})$ \\ \cline{2-3}
& Lower bound \citep{Ouyang-2019-Lower} & $\tilde{\Omega}(\sqrt{\kappax/\epsilon})$ \\ \hline \hline
\multirow{4}{*}{Convex-Concave} & \citet{Nemirovski-2004-Prox} & \multirow{3}{*}{$O(\epsilon^{-1})$} \\ \cline{2-2}
& \citet{Nesterov-2007-Dual} & \\ \cline{2-2}
& \citet{Tseng-2008-Accelerated} & \\ \hhline{|~--|}
& \cellcolor{light-gray} \textbf{This paper} (Corollary \ref{cor:cc}) & \cellcolor{light-gray} $\tilde{O}(\epsilon^{-1})$ \\  \cline{2-3}
& Lower bound \citep{Ouyang-2019-Lower} & $\Omega(\epsilon^{-1})$ \\ \hline 
\end{tabular}
\end{table*}
We propose the first near-optimal algorithms for solving the strongly-convex-strongly-concave and strongly-convex-concave minimax optimization problems. In the former setting, our algorithm achieves a gradient complexity of $\tilde{O}(\sqrt{\kappax\kappay})$ which matches the lower complexity bound~\citep{Ibrahim-2019-Lower, Zhang-2019-Lower} up to logarithmic factors. In the latter setting, our algorithm attains a gradient complexity of $\tilde{O}(\sqrt{\kappax/\epsilon})$ which again matches the lower complexity bound~\citep{Ouyang-2019-Lower} up to logarithmic factors. In addition, our algorithm extends to the general convex-concave setting, achieving a gradient complexity of $\tilde{O}(\epsilon^{-1})$, which matches the lower bound of \citet{Ouyang-2019-Lower} as well as the best existing upper bounds \citep{Nemirovski-2004-Prox, Nesterov-2007-Dual, Tseng-2008-Accelerated} up to logarithmic factors.  

Our second contribution is a class of accelerated algorithms for the smooth nonconvex-strongly-concave and nonconvex-concave minimax optimization problems. In the former setting, our algorithm achieves a gradient complexity bound of $\tilde{O}(\sqrt{\kappay}\epsilon^{-2})$ which improves the best known bound $\tilde{O}(\kappay^2\epsilon^{-2})$~\citep{Jin-2019-Minmax, Rafique-2018-Non, Lin-2019-Gradient, Lu-2019-Hybrid}. In the latter setting, our algorithms specialize to a range of different notions of optimality. In particular, expressing our results in terms of stationarity of $f$, our algorithm achieves a gradient complexity bound of $\tilde{O}(\epsilon^{-2.5})$, which improves the best known bound $\tilde{O}(\epsilon^{-3.5})$~\citep{Nouiehed-2019-Solving}. In terms of stationarity of the function $\Phi(\cdot):= \max_{\y\in\YCal} f(\cdot, \y)$, our algorithm achieves a gradient complexity bound of $\tilde{O}(\epsilon^{-3})$ which matches the current state-of-the-art results~\citep{Thekumparampil-2019-Efficient, Kong-2019-Accelerated}. 

We provide a head-to-head comparison between our results and existing results in the literature in Table~\ref{tab:results_cc} for convex-concave settings, and Table \ref{tab:results_nc} for nonconvex-concave settings.
\begin{table*}[!t]
\renewcommand{\arraystretch}{1.4}
\centering
\caption{Comparison of gradient complexities to find an $\epsilon$-stationary point of $f$ (Definition \ref{Def:eps_stationary}) or $\epsilon$-stationary point of $\Phi(\cdot):=\max_{\y\in\YCal} f(\cdot, \y)$ (Definition \ref{Def:NSC-Stationary}, \ref{Def:NC-Stationary}) in the nonconvex-concave settings. This table only highlights the dependence on tolerance $\epsilon$ and the condition number $\kappay$.}
\label{tab:results_nc}
\begin{tabular}{|c|c|c|} \hline
\textbf{Settings} & \textbf{References} & \textbf{Gradient Complexity} \\ \hline 
\multirow{5}{*}{\shortstack{Nonconvex-Strongly-Concave\\{\footnotesize(stationarity of $f$ or}\\{\footnotesize stationarity of $\Phi$)}}} & \citet{Jin-2019-Minmax} & \multirow{4}{*}{$\tilde{O}(\kappay^2\epsilon^{-2})$} \\ \cline{2-2}
& \citet{Rafique-2018-Non} & \\ \cline{2-2}
& \citet{Lin-2019-Gradient} & \\ \cline{2-2} 
& \citet{Lu-2019-Hybrid} & \\ \hhline{|~--|}
& \cellcolor{light-gray} \textbf{This paper} (Theorem \ref{thm:nsc} \& \ref{thm:nsc-Moreau}) & \cellcolor{light-gray} $\tilde{O}(\sqrt{\kappay}\epsilon^{-2})$ \\ \hline \hline
\multirow{3}{*}{\shortstack{Nonconvex-Concave\\{\footnotesize(stationarity of $f$)}}} & \citet{Lu-2019-Hybrid}  & $\tilde{O}(\epsilon^{-4})$ \\ \cline{2-3}
& \citet{Nouiehed-2019-Solving} & $\tilde{O}(\epsilon^{-3.5})$ \\ \hhline{|~--|}
& \citet{Ostrovskii-2020-Efficient} & $\tilde{O}(\epsilon^{-2.5})$ \\ \hhline{|~--|}
& \cellcolor{light-gray} \textbf{This paper} (Corollary \ref{cor:nc}) & \cellcolor{light-gray} $\tilde{O}(\epsilon^{-2.5})$ \\ \hline
\multirow{6}{*}{\shortstack{Nonconvex-Concave\\{\footnotesize(stationarity of $\Phi$)}}} & \citet{Jin-2019-Minmax}  & \multirow{3}{*}{$\tilde{O}(\epsilon^{-6})$} \\ \cline{2-2} 
& \citet{Rafique-2018-Non} & \\ \cline{2-2}
& \citet{Lin-2019-Gradient} & \\ \cline{2-3} 
& \citet{Thekumparampil-2019-Efficient} & \multirow{2}{*}{$\tilde{O}(\epsilon^{-3})$} \\ \cline{2-2} 
& \citet{Zhao-2020-Primal} & \\ \hhline{|~--|}
& \cellcolor{light-gray} \textbf{This paper} (Corollary \ref{cor:nc-Moreau}) & \cellcolor{light-gray} $\tilde{O}(\epsilon^{-3})$ \\ \hline
\end{tabular}
\end{table*}


\section{Related work}\label{sec:related_work}
To the best of our knowledge, the earliest algorithmic schemes for solving the bilinear minimax problem, $\min_{\x \in \Delta^m} \max_{\y \in \Delta^n} \x^\top A \y$, date back to Brown's fictitious play~\citep{Brown-1951-Iterative} and Dantzig's simplex method~\citep{Dantzig-1963-Linear}. This problem can also be solved by Korpelevich's extragradient (EG) algorithm~\citep{Korpelevich-1976-Extragradient}, which can be shown to be linearly convergent when $A$ is square and full rank~\citep{Tseng-1995-Linear}. There are also several recent papers studying the convergence of EG and its variants; see~\citet{Chambolle-2011-First, Malitsky-2015-Projected, Yadav-2018-Stabilizing} for reflected gradient descent ascent,~\citet{Daskalakis-2018-Training, Mokhtari-2019-Unified, Mokhtari-2019-Proximal} for optimistic gradient descent ascent (OGDA) and~\citet{Rakhlin-2013-Online, Rakhlin-2013-Optimization, Mertikopoulos-2018-Optimistic, Chavdarova-2019-Reducing, Hsieh-2019-Convergence, Mishchenko-2019-Revisiting} for other variants. In the bilinear setting, \citet{Daskalakis-2018-Training} established the convergence of the optimistic gradient descent ascent (OGDA) method to a neighborhood of the solution; \citet{Liang-2019-Interaction} proved the linear convergence of the OGDA algorithm using a dynamical system approach. Very recently, \citet{Mokhtari-2019-Unified} have proposed a unified framework for achieving the sharpest convergence rates of both EG and OGDA algorithms. 

For the convex-concave minimax problem, \citet{Nemirovski-2004-Prox} proved that his mirror-prox algorithm returns an $\epsilon$-saddle point within the gradient complexity of $O(\epsilon^{-1})$ when $\XCal$ and $\YCal$ are bounded. This algorithm was subsequently generalized by \citet{Auslender-2005-Interior} to a class of distance-generating functions, and the complexity result was extended to unbounded sets and composite objectives~\citep{Monteiro-2010-Complexity, Monteiro-2011-Complexity} using the hybrid proximal extragradient algorithm with different error criteria. \citet{Nesterov-2007-Dual} developed a dual extrapolation algorithm which possesses the same complexity bound as in~\citet{Nemirovski-2004-Prox}. Later on,~\citet{Tseng-2008-Accelerated} presented a unified treatment of these algorithms and a refined convergence analysis with same complexity result.~\citet{Nedic-2009-Subgradient} analyzed the (sub)gradient descent ascent algorithm for convex-concave saddle point problems when the (sub)gradients are bounded over the constraint sets.~\citet{Abernethy-2019-Last} presented a Hamiltonian gradient descent algorithm with last-iterate convergence under a ``sufficiently bilinear" condition. 

Several papers have studied special cases in the convex-concave setting. For the special case when the objective function is a composite bilinear form, $f(\x, \y) = g(\x) + \x^\top A\y - h(\y)$,~\citet{Chambolle-2011-First} introduced a primal-dual algorithm that converges to a saddle point with the rate of $O(1/\epsilon)$ when the convex functions $g$ and $h$ are smooth.~\citet{Nesterov-2005-Smooth} proposed a smoothing technique and proved that the resulting algorithm achieves an improved rate with better dependence on Lipschitz constant of $\nabla g$ when $h$ is the convex and smooth function and $\XCal, \YCal$ are both bounded.~\citet{He-2016-Accelerated} and~\citet{Kolossoski-2017-Accelerated} proved that such result also hold when $\XCal, \YCal$ are unbounded or the space is non-Euclidean.~\citet{Chen-2014-Optimal, Chen-2017-Accelerated} generalized Nesterov's technique to develop optimal algorithms for solving a class of stochastic saddle point problems and stochastic monotone variational inequalities. For a class of certain purely bilinear games where $g$ and $h$ are zero functions,~\citet{Azizian-2020-Accelerating} demonstrated that linear convergence is possible for several algorithms and their new algorithm achieved the tight bound. The second case is the so-called affinely constrained smooth convex problem, i.e., $\min_{\x \in \XCal} g(\x), \st \ A\x = \su$.~\citet{Esser-2010-General} proposed a $O(\epsilon^{-1})$ primal-dual algorithm while~\citet{Lan-2016-Iteration} provided a first-order augmented Lagrangian method with the same $O(\epsilon^{-1})$ rate. By exploiting the structure,~\citet{Ouyang-2015-Accelerated} proposed a near-optimal algorithm in this setting. 

For the strongly convex-concave minimax problem,~\citet{Tseng-1995-Linear} and~\citet{Nesterov-2006-Solving} proved that their algorithms find an $\epsilon$-saddle point with a gradient complexity of $\tilde{O}(\kappax + \kappay)$ using a variational inequality. Using a different approach,~\citet{Gidel-2018-A} and~\citet{Mokhtari-2019-Unified} derived the same complexity results for the OGDA algorithm. Very recently,~\citet{Alkousa-2019-Accelerated} proposed an accelerated gradient sliding algorithm with a gradient complexity of $\tilde{O}(\min\{\kappax\sqrt{\kappay}, \kappay\sqrt{\kappax}\})$ while~\citet{Ibrahim-2019-Lower} and~\citet{Zhang-2019-Lower} established a lower complexity bound of $\tilde{\Omega}(\sqrt{\kappax\kappay})$ among all the first-order algorithms in this setting. 

For strongly-convex-concave minimax problems, the best known general lower bound for first-order algorithm is $O(\sqrt{\kappax/\epsilon})$, as shown by~\citet{Ouyang-2019-Lower}. Several papers have studied strongly-convex-concave minimax problem with additional structures. This includex optimizing a strongly convex function with linear constraints~\citep{Goldstein-2014-Fast, Xu-2018-Accelerated, Xu-2019-Iteration}, the case when $\x$ and $\y$ are connected only through a bilinear term $\x^\top A \y$~\citep{Nesterov-2005-Smooth, Chambolle-2016-Ergodic, Xie-2019-Accelerated} and the case when $f(\x, \cdot)$ is linear for each $\x \in \br^m$~\citep{Juditsky-2011-First, Hamedani-2018-Primal, Zhao-2019-Optimal}. The algorithms developed in these works were all guaranteed to return an $\epsilon$-saddle point with a gradient complexity of $\tilde{O}(1/\sqrt{\epsilon})$ and some of them even achieve a near-optimal gradient complexity of $\tilde{O}(\sqrt{\kappax/\epsilon})$~\citep{Nesterov-2005-Smooth, Chambolle-2016-Ergodic}. However, the best known upper complexity bound for general strongly-convex-concave minimax problems is $O(\kappax/\sqrt{\epsilon})$ which was shown using the \textit{dual implicit accelerated gradient} algorithm~\citep{Thekumparampil-2019-Efficient}.

For nonconvex-concave minimax problems, a line of recent work~\citep{Jin-2019-Minmax, Rafique-2018-Non, Lin-2019-Gradient} has studied various algorithms and proved that they can find an approximate stationary point of $\Phi(\cdot) :=\max_{\y\in\YCal} f(\cdot, \y)$. In a deterministic setting, all of these algorithms guarantee a rate of $\tilde{O}(\kappay^2\epsilon^{-2})$ and $\tilde{O}(\epsilon^{-6})$ when $f(\x, \cdot)$ is strongly concave and concave respectively.~\citet{Thekumparampil-2019-Efficient} consider the same setting as ours and proposed a proximal dual implicit accelerated gradient algorithm and proved that it finds an approximate stationary point of $\Phi(\cdot)$ with the total gradient complexity of $\tilde{O}(\epsilon^{-3})$.~\citet{Kong-2019-Accelerated} consider a general nonconvex minimax optimization model: $\min_\x h(\x) + \rho(\x)$, where $h$ is a ``simple" proper, lower semi-continuous and convex function and $\rho(\x) = \max_{\y \in \YCal} f(\x, \y)$ with $f$ satisfying that $-f(\x, \cdot)$ is proper, convex, and lower semi-continuous. They propose to smooth $\rho$ to $\rho_\xi(\x) = \max_{\y \in \YCal} f(\x, \y) - (1/2\xi)\|\y - \y_0\|^2$ and apply an accelerated inexact proximal point method to solve the smoothed problem $\min_\x h(\x) + \rho_\xi(\x)$. The resulting AIPP-S algorithm attains the iteration complexity of $O(\epsilon^{-3})$ using a slightly different but equivalent notion of stationarity but requires the \textit{exact} gradient of $\rho_\xi$ at each iteration. This amounts to assuming that $\max_{\y \in \YCal} f(\x, \y) - (1/2\xi)\|\y - \y_0\|^2$ can be solved exactly, which is restrictive due to the potentially complicated structure of $f(\x, \cdot)$ or $\YCal$. If $f$ is further assumed to be smooth,~\citet{Zhao-2020-Primal} developed a variant of AIPP-S algorithm which only requires an inexact gradient of $\rho_\xi$ at each iteration and attains the total gradient complexity of $\tilde{O}(\epsilon^{-3})$. On the other hand, the stationarity of $f(\cdot, \cdot)$ is proposed for quantifying the efficiency in nonconvex-concave minimax optimization~\citep{Lu-2019-Hybrid, Nouiehed-2019-Solving, Kong-2019-Accelerated, Ostrovskii-2020-Efficient}. Using this notion of stationarity,~\citet{Kong-2019-Accelerated} attains the rate of $O(\epsilon^{-2.5})$ but requires the \textit{exact} gradient of $\rho_\xi$ at each iteration. Without this assumption, the current state-of-the-art rate is $\tilde{O}(\epsilon^{-2.5})$ achieved by our Algorithm~\ref{alg:Minimax_PG} and the algorithm proposed by a concurrent work~\citep{Ostrovskii-2020-Efficient}. Both algorithms are based on constructing an auxiliary function $f_{\epsilon, \y}$ and applying an accelerated solver for minimax proximal steps. Finally, several other algorithms have been developed either for specific nonconvex-concave minimax problems or in stochastic setting; see~\citet{Namkoong-2016-Stochastic, Sinha-2018-Certifiable, Sanjabi-2018-Solving, Grnarova-2018-An} for the details. 


\section{Preliminaries}\label{sec:setup}
In this section, we clarify the notation used in this paper, review some background and provide formal definitions for the class of functions and optimality measure considered in this paper. 
\paragraph{Notation.} We use bold lower-case letters to denote vectors, as in $\x, \y, \z$ and calligraphic upper case letters to denote sets, as in $\XCal$ and $\YCal$. For a differentiable function $f(\cdot): \br^n \rightarrow \br$, we le4t $\grad f(\z)$ denote the gradient of $f$ at $\z$. For a function $f(\cdot, \cdot): \br^m \times \br^n \rightarrow \br$ of two variables, $\gradx f(\x, \y)$ (or $\grady f(\x, \y)$) to denote the partial gradient of $f$ with respect to the first variable (or the second variable) at point $(\x, \y)$. We also use $\grad f(\x, \y)$ to denote the full gradient at $(\x, \y)$ where $\grad f(\x, \y) = (\gradx f(\x, \y), \grady f(\x, \y))$. For a vector $\x$, we denote $\|\x\|$ as its $\ell_2$-norm. For constraint sets $\XCal$ and $\YCal$, we let $\Dx$ and $\Dy$ denote their diameters, where $\Dx = \max_{\x, \x'\in \XCal} \norm{\x - \x'}$ and $\Dy = \max_{\y, \y' \in \YCal} \norm{\y - \y'}$. We use the notation $\PX$ and $\PY$ to denote projections onto the sets $\XCal$ and $\YCal$. Finally, we use the notation $O(\cdot), \Omega(\cdot)$ to hide only absolute constants which do not depend on any problem parameter, and notation $\tilde{O}(·), \tilde{\Omega}(·)$ to hide only absolute constants and log factors. 

\subsection{Minimax optimization}
We are interested in the $\ell$-smooth minimax optimization problems in the form~\eqref{prob:minimax}. The regularity conditions that we consider for the function $f$ are as follows.   
\begin{definition}\label{Def:Lipschitz}
A function $f$ is $L$-Lipschitz if for $\forall \z, \z' \in \br^n$, that $|f(\z) - f (\z')| \leq L\|\z - \z'\|$.
\end{definition}
\begin{definition}\label{Def:smoothness}
A function $f$ is $\ell$-smooth if for $\forall \z, \z' \in \br^n$, that $\|\grad f(\z) - \grad f (\z')\| \leq \ell\|\z - \z'\|$.
\end{definition}
\begin{definition}
A differentiable function $\phi: \br^d \rightarrow \br$ is $\mu$-strongly-convex if for any $\x', \x \in \br^d$:
\begin{equation*}
\phi(\x') \geq \phi(\x) + (\x' - \x)^\top\nabla\phi(\x) + (\mu/2)\|\x' - \x\|^2
\end{equation*}
Furthermore, $\phi$ is $\mu$-strongly-concave if $-\phi$ is $\mu$-strongly-convex. If we set $\mu=0$, then we recover the definitions of convexity and concavity for a continuous differentiable function.
\end{definition}

\paragraph{Convex-concave setting:} we assume that $f(\cdot, \y)$ is convex for each $\y \in \YCal$ and $f(\x, \cdot)$ is concave for each $\x \in \XCal$. \textit{Here $\XCal$ and $\YCal$ are both convex and bounded}. Under these conditions, the Sion's minimax theorem~\citep{Sion-1958-General} guarantees that
\begin{equation}\label{Def:Sion}
\max_{\y \in \YCal} \min_{\x \in \XCal} f(\x, \y) \ = \ \min_{\x \in \XCal} \max_{\y \in \YCal} f(\x, \y). 
\end{equation}
Furthermore, there exists at least one \textbf{saddle point (or Nash equilibrium)} $(\x^\star, \y^\star) \in \XCal \times \YCal$ such that the following equality holds true:
\begin{equation}\label{Def:Nash}
 \min_{\x \in \XCal} f(\x, \y^\star) \ = \ f(\x^\star, \y^\star) \ = \ \max_{\y \in \YCal} f(\x^\star, \y). 
\end{equation}
Therefore, for any point $(\hat{\x}, \hat{\y}) \in \XCal \times \YCal$, the duality gap $\max_{\y \in \YCal} f(\hat{\x}, \y) - \min_{\x \in \XCal} f(\x, \hat{\y})$ forms the basis for a standard optimality criterion. Formally, we define
\begin{definition}\label{Def:eps_saddle}
A point $(\hat{\x}, \hat{\y}) \in \XCal \times \YCal$ is an \textbf{$\epsilon$-saddle point} of a convex-concave function $f(\cdot, \cdot)$ if $\max_{\y \in \YCal} f(\hat{\x}, \y) - \min_{\x \in \XCal} f(\x, \hat{\y}) \leq \epsilon$. If $\epsilon = 0$, then $(\hat{\x}, \hat{\y})$ is a saddle point. 
\end{definition}

In the case when $f(\cdot, \y)$ is strongly convex for each $\y \in \YCal$ and $f(\x, \cdot)$ is strongly concave for each $\x \in \XCal$, we refer $\mu_\x$ and $\mu_\y$ to the strongly-convex or strongly-concave module. If $f$ is further $\ell$-smooth, we denote $\kappax = \ell/\mu_\x$ and $\kappay = \ell/\mu_\y$ as the condition numbers of $f(\cdot, \y)$ and $f(\x, \cdot)$. 

\paragraph{Nonconvex-concave setting:} we only assume that $f(\x, \cdot)$ is concave for each $\x \in \mathbb{R}^m$. The function $f(\cdot, \y)$ can be possibly nonconvex for some $\y \in \YCal$. \textit{Here $\XCal$ is convex but possibly unbounded while $\YCal$ is convex and bounded.} In general, finding a global Nash equilibrium of $f$ is intractable since in the special case where $\YCal$ has only a single element, this problem reduces to a nonconvex optimization problem in which finding a global minimum is already NP-hard~\citep{Murty-1987-Some}. Similar to the literature in nonconvex constrained optimization, we opt to find local surrogates---stationary points---whose gradient mappings are zero. Formally, we define our optimality criterion as follows. 
\begin{definition}\label{Def:eps_stationary}
A point $(\hat{\x}, \hat{\y}) \in \XCal \times \YCal$ is an \textbf{$\epsilon$-stationary point} of an $\ell$-smooth function $f(\cdot, \cdot)$ if 
\begin{equation*}
\ell\|\PX[\hat{\x} - (1/\ell)\gradx f(\hat{\x}, \hat{\y}^+)] - \hat{\x}\| \leq \epsilon, \quad \ell\|\hat{\y}^+ - \hat{\y}\| \leq \epsilon. 
\end{equation*}
where 
\begin{equation*}
\hat{\y}^+ = \PY[\hat{\y} + (1/\ell)\grady f(\hat{\x}, \hat{\y})]. 
\end{equation*}
If $\epsilon = 0$, then $(\hat{\x}, \hat{\y})$ is a stationary point.
\end{definition}
In the absence of the constraint set $\XCal$, Definition~\ref{Def:eps_stationary} reduces to the standard condition $\|\gradx f(\hat{\x}, \hat{\y}^+)\| \leq \epsilon$ and $\ell\|\hat{\y}^+ - \hat{\y}\| \leq \epsilon$ which is consistent with~\citet[Definition~4.10]{Lin-2019-Gradient}. Intuitively, the quantity $\|\PY[\hat{\y} + (1/\ell)\grady f(\hat{\x}, \hat{\y})] - \hat{\y}\|$ represents the distance between a point $\hat{\y}$ and a point obtained by performing one-step projected partial gradient ascent at a point $(\hat{\x}, \hat{\y})$ starting from a point $\hat{\y}$. It also refers to the norm of gradient mapping at $(\hat{\x}, \hat{\y})$; see~\citet{Nesterov-2013-Gradient} for the details. 

We note that this notion of stationarity of $f$ (Definition \ref{Def:eps_stationary}) is closely related to an optimality notion in terms of stationary points of the function $\Phi(\cdot):=\max_{\y \in \YCal} f(\cdot, \y)$ for nonconvex-concave functions. We refer readers to Appendix \ref{sec:envelope} for more discussion.
\begin{algorithm}[t]
\caption{$\textsc{AGD}(g, \XCal, \x_0, \ell, \mu, \epsilon)$}\label{alg:AGD}
\begin{algorithmic}[1]
\STATE \textbf{Input:} initial point $\x_0 \in \XCal$, smoothness $\ell$, strongly-convex module $\mu$ and tolerance $\epsilon>0$. 
\STATE \textbf{Initialize:} set $t \leftarrow 0$, $\tilde{\x}_0 \leftarrow \x_0$, $\eta \leftarrow 1/\ell$, $\kappa \leftarrow \ell/\mu$ and $\theta \leftarrow \frac{\sqrt{\kappa}-1}{\sqrt{\kappa}+1}$. 
\REPEAT 
\STATE $t \leftarrow t + 1$
\STATE $\x_{t} \leftarrow \proj_{\XCal}[\tilde{\x}_{t-1} - \eta\grad g(\tilde{\x}_{t-1})]$. \label{line:AGD_gradient}
\STATE $\tilde{\x}_{t} \leftarrow \x_{t} + \theta(\x_{t} - \x_{t-1})$.  \label{line:AGD_momentum}
\UNTIL{$\|\x_t - \proj_\XCal(\x_t - \eta\grad g(\x_t))\|^2 \leq \frac{\epsilon}{2\kappa^2(\ell-\mu)}$ is satisfied.} 
\label{line:AGD_term}
\STATE \textbf{Output:} $\proj_\XCal(\x_t - \eta\grad g(\x_t))$. 
\end{algorithmic}
\end{algorithm}

\subsection{Nesterov's accelerated gradient descent}
Nesterov's Accelerated Gradient Descent (AGD) dates back to the seminal paper~\citep{Nesterov-1983-Method} where it is shown to be optimal among all the first-order algorithms for smooth and convex functions~\citep{Nesterov-2018-Lectures}. We present a version of AGD in Algorithm~\ref{alg:AGD} which is frequently used to minimize an $\ell$-smooth and $\mu$-strongly convex function $g$ over a convex set $\XCal$. The key steps of the AGD algorithm are Line \ref{line:AGD_gradient}-\ref{line:AGD_momentum}, where Lines \ref{line:AGD_gradient} performs a projected gradient descent step, while Line \ref{line:AGD_momentum} performs a momentum step, which ``overshoots'' the iterate in the direction of momentum $(\x_t - \x_{t-1})$. Line~\ref{line:AGD_term} is the stopping condition to ensure that the output achieves the desired optimality.

The following theorem provides an upper bound on the gradient complexity of AGD; i.e., the total number of gradient evaluations to find an $\epsilon$-optimal point in terms of function value. 
\begin{theorem}\label{Theorem:AGD}
Assume that $g$ is $\ell$-smooth and $\mu$-strongly convex, the output $\hat{\x} = \textsc{AGD}(g, \x_0, \ell, \mu, \epsilon)$ satisfies $g(\hat{\x}) \leq \min_{\x\in \XCal} g(\x) + \epsilon$ and the total number of gradient evaluations is bounded by
\begin{equation*}
O\left(\sqrt{\kappa}\log\left(\frac{\kappa^3\ell\|\x_0 - \x^\star\|^2}{\epsilon}\right) \right), 
\end{equation*}
where $\kappa = \ell/\mu$ is the condition number, and $\x^\star \in \XCal$ is the unique global minimum of $g$.
\end{theorem}
Compared with the classical result for Gradient Descent (GD), which requires $\tilde{O}(\kappa)$ gradient evaluations in the same setting, AGD improves over GD by a factor of $\sqrt{\kappa}$. AGD will be used as a basic component for acceleration in this paper.

\section{Algorithm Components} \label{sec:component}
In this section, we present two main algorithm components. Both of them are crucial for our final algorithms to achieve near-optimal convergence rates.
\begin{algorithm}[t]
\caption{$\textsc{Inexact-APPA}(g, \x_0, \ell, \mu, \epsilon, T)$}\label{alg:APPA}
\begin{algorithmic}[1]
\STATE \textbf{Input:} initial point $\x_0 \in \XCal$, proximal parameter $\ell$, strongly-convex module $\mu$, tolerance $\epsilon > 0$ and the maximum iteration number $T > 0$. 
\STATE \textbf{Initialize:} set $\tilde{\x}_0 \leftarrow \x_0$, $\kappa \leftarrow \frac{\ell}{\mu}$, $\delta \leftarrow \frac{\epsilon}{(10\kappa)^2}$ and $\theta \leftarrow \frac{2\sqrt{\kappa}-1}{2\sqrt{\kappa}+1}$. 
\FOR{$t = 1, \cdots, T$}
\STATE find $\x_{t}$ so that $g(\x_t) + \ell\|\x_t - \tilde{\x}_{t-1}\|^2 \leq \min_{\x\in\XCal} \{g(\x) + \ell\|\x - \tilde{\x}_{t-1}\|^2\} + \delta$.  \label{line:proximal}
\STATE $\tilde{\x}_{t} \leftarrow \x_{t} + \theta(\x_{t} - \x_{t-1})$. 
\ENDFOR
\STATE \textbf{Output:} $\x_T$. 
\end{algorithmic}
\end{algorithm}

\subsection{Inexact Accelerated Proximal Point Algorithm}
Our first component is the Accelerated Proximal Point Algorithm (APPA, Algorithm \ref{alg:APPA}) for minimizing a function $g(\cdot)$. Comparing APPA with classical AGD (Algorithm \ref{alg:AGD}), we note that both of them have momentum steps which yield acceleration. The major difference is in Line \ref{line:proximal} of Algorithm \ref{alg:APPA}, where APPA solves a proximal subproblem
\begin{equation} \label{eq:proximal}
\x_t \leftarrow \argmin_{\x\in\XCal} \ g(\x) + \ell \|\x - \tilde{\x}_{t-1}\|^2. 
\end{equation}
instead of performing a gradient-descent step as in AGD (Line \ref{line:AGD_gradient} in Algorithm \ref{alg:AGD}). We refer to the parameter $\ell$ in \eqref{eq:proximal} as the \emph{proximal parameter}.

We present an inexact version in Algorithm \ref{alg:APPA} where we tolerate a small error $\delta$ in terms of the function value in solving the proximal subproblem \eqref{eq:proximal}. That is, the solution $\x_t$ satisfies
\begin{equation*}
g(\x_t) + \ell\|\x_t - \tilde{\x}_{t-1}\|^2 \leq \min_{\x\in\XCal} \{g(\x) + \ell\|\x - \tilde{\x}_{t-1}\|^2\} + \delta.
\end{equation*}
A theoretical guarantee for the inexact APPA algorithm is presented in the following theorem, which claims that as long as $\delta$ is sufficiently small, the algorithm finds an $\epsilon$-optimal point of any $\mu$-strongly-convex function $g$ with proximal parameter $\ell$ in $\tilde{O}(\sqrt{\ell/\mu})$ iterations.
\begin{theorem}\label{Theorem:inexact-APPA}
Assume that $g$ is $\mu$-strongly convex, $\epsilon \in (0, 1)$ and $\ell>\mu$. There exists $T > 0$ such that the output $\hat{\x} = \textsc{Inexact-APPA}(g, \x_0, \ell, \mu, \epsilon, T)$ satisfies $g(\hat{\x}) \leq \min_{\x \in \XCal} g(\x) + \epsilon$ and $T > 0$ satisfies the following inequality, 
\begin{equation*}
T \geq \ c\sqrt{\kappa}\log\left(\frac{g(\x_0) - g(\x^\star) + (\mu/4)\|\x_0 - \x^\star\|^2}{\epsilon}\right),  
\end{equation*}
where $\kappa = \ell/\mu$ is an effective condition number, $\x^\star \in \XCal$ is the unique global minimum of $g$, and $c > 0$ is an absolute constant. 
\end{theorem}
Comparing with Theorem \ref{Theorem:AGD}, the most important difference here is that Theorem \ref{Theorem:inexact-APPA} does not require the function $g$ to have any smoothness property. In fact, $\ell$ is only a proximal parameter in proximal subproblem \eqref{eq:proximal}, which does not necessarily relate to the smoothness of $g$. On the flip side, the proximal subproblem~\eqref{eq:proximal} can not be easily solved in general. Theorem \ref{Theorem:inexact-APPA} guarantees the iteration complexity of Algorithm \ref{alg:AGD} while the complexity for solving these proximal steps is not discussed. 

We conclude that APPA has a unique advantage over AGD in settings where $g$ does not have a smoothness property but the proximal step \eqref{eq:proximal} is easy to solve. These settings include LASSO \citep{Beck-2009-Fast}, as well as minimax optimization problems (as we show in later sections).
\begin{algorithm}[!t]
\caption{$\textsc{\AGPROX}(g, \x_0, \y_0, \ell, \mu_\x, \mu_\y, \epsilon) $}\label{alg:Proxy}
\begin{algorithmic}[1]
\STATE \textbf{Input:} initial point $\x_0, \y_0$, smoothness $\ell$, strongly convex module $\mu_\x, \mu_\y$ and tolerance $\epsilon>0$.
\STATE \textbf{Initialize:} $t \leftarrow 0$, $\tilde{\x}_0 \leftarrow \x_0$, $\eta \leftarrow \frac{1}{2\kappax\ell}$, $\kappax \leftarrow \frac{\ell}{\mu_\x}$, $\kappay \leftarrow \frac{\ell}{\mu_\y}$, $\theta \leftarrow \frac{4\sqrt{\kappax\kappay}-1}{4\sqrt{\kappax\kappay}+1}$, $\tilde{\epsilon} \leftarrow \frac{\epsilon}{(10\kappax\kappay)^{7}}$. 
\REPEAT 
\STATE $t \leftarrow t+ 1$.
\STATE $\tilde{\x}_{t-1}\leftarrow \textsc{AGD}(g(\cdot, \tilde{\y}_{t-1}), \x_0, \ell, \mu_{\x}, \tilde{\epsilon})$. \label{line:AG2_start}
\STATE $\y_{t} \leftarrow \proj_\YCal[\tilde{\y}_{t-1} + \eta\grady g(\tilde{\x}_{t-1}, \tilde{\y}_{t-1})]$. 
\STATE $\tilde{\y}_{t} \leftarrow \y_{t} + \theta(\y_{t} - \y_{t-1})$. \label{line:AG2_end}
\STATE $\x_{t}\leftarrow \textsc{AGD}(g(\cdot, \y_t), \x_0, \ell, \mu_{\x}, \tilde{\epsilon})$. \label{line:AG2_term_start}
\UNTIL{$\|\y_t - \proj_\YCal(\y_t + \eta\grady g(\x_t, \y_t))\|^2 \leq \frac{\epsilon}{(10\kappax\kappay)^4 \ell}$ is satisfied.} \label{line:AG2_term_end}
\STATE \textbf{Output:} $\proj_\XCal(\x_t - (1/2\kappa_\y\ell)\gradx g(\x_t, \y_t))$. 
\end{algorithmic}
\end{algorithm}

\subsection{Accelerated Solver for Minimax Proximal Steps} \label{sec:subsolver}
In minimax optimization problems of the form \eqref{prob:minimax}, we are interested in solving the following proximal subproblem as follows, 
\begin{equation}\label{eq:prox_minimax_original_form}
\x_{t+1} \leftarrow \argmin_{\x\in\XCal} \Phi(\x) + \ell \norm{\x - \tilde{\x}}^2, \quad \textnormal{where } \ \Phi(\x) := \max_{\y\in\YCal} f(\x, \y),
\end{equation}
which is equivalent to solving the following minimax problem:
\begin{equation} \label{eq:prox_minimax}
\min_{\x\in\XCal} \max_{\y\in\YCal} ~ \tilde{g}(\x, \y) := f(\x, \y) + \ell \norm{\x - \tilde{\x}}^2.
\end{equation}
For a generic strongly-convex-strongly-concave function $g(\cdot, \cdot)$, solving a minimax problem is equivalent to solving a maximin problem, due to Sion's minimax theorem:
\begin{equation*}
\min_{\x \in \XCal} \max_{\y \in \YCal} g(\x, \y) = \max_{\y \in \YCal} \min_{\x \in \XCal} g(\x, \y).
\end{equation*}
A straightforward way of solving the maximin problem is to use a double-loop algorithm which solves the maximization and minimization problems on two different time scales. Specifically, the inner loop performs AGD on function $g(\cdot, \y)$ to solve the inner minimization; i.e., to compute $\Psi(\y) := \min_{\x\in \XCal} g(\x, \y)$ for each $\y$, and the outer loop performs Accelerated Gradient Ascent (AGA) on the function $\Psi(\cdot)$ to solve the outer maximization. Since the algorithm aims to solve a maximin problem we use AGA-AGD, and we name the algorithm \textsc{\AGPROX}. See Algorithm \ref{alg:Proxy} for the formal version of this algorithm. We also incorporate Lines \ref{line:AG2_term_start}-\ref{line:AG2_term_end} to check termination conditions, which ensures that the output achieves the desired optimality. The theoretical guarantee for Algorithm \ref{alg:Proxy} is given in the following theorem.
\begin{theorem}\label{Theorem:maximin-AG2}
Assume that $g(\cdot, \cdot)$ is $\ell$-smooth, $g(\cdot, \y)$ is $\mu_\x$-strongly convex for each $\y \in \YCal$ and $g(\x, \cdot)$ is $\mu_\y$-strongly concave for each $\x \in \XCal$. Then $\hat{\x} = \textsc{\AGPROX}(g, \x_0, \y_0, \ell, \mu_\x, \mu_\y, \epsilon)$ satisfies that $\max_{\y\in\YCal} g(\hat{\x}, \y) \leq \min_{\x \in \XCal} \max_{\y \in \YCal} g(\x, \y) + \epsilon$ and the total number of gradient evaluations is bounded by
\begin{equation*}
O\left(\kappax\sqrt{\kappay}\cdot \log^2\left(\frac{(\kappax+\kappay)\ell(\tilde{D}_\x^2 + \Dy^2)}{\epsilon}\right)\right),
\end{equation*}
where $\kappa_\x = \ell/\mu_\x$ and $\kappa_\y = \ell/\mu_\y$ are condition numbers, $\tilde{D}_\x = \|\x_0 - \x_g^\star(\y_0)\|$ is the initial distance where $\x_g^\star(\y_0) = \argmin_{\x \in \XCal} g(\x, \y_0)$ and $\Dy > 0$ is the diameter of the constraint set $\YCal$.
\end{theorem}
Theorem~\ref{Theorem:maximin-AG2} claims that Algorithm~\ref{alg:Proxy} finds an $\epsilon$-optimal point in $\tilde{O}(\kappa_\x \sqrt{\kappa_\y})$ iterations for strongly-convex-strongly-concave functions. This rate does not match the lower bound $\tilde{\Omega}(\sqrt{\kappa_\x\kappa_\y})$~\citep{Ibrahim-2019-Lower, Zhang-2019-Lower}. At a high level, it takes AGD $\tilde{O}(\sqrt{\kappa_{\x}})$ steps to solve the inner minimization problem and compute $\Psi(\y) := \min_{\x\in \XCal} g(\x, \y)$. Despite the fact that the function $g$ is $\ell$-smooth, function $\Psi$ is only guaranteed to be $(\kappa_\x \ell)$-smooth in the worst case, which makes the condition number of $\Psi$ be $\kappa_{\x}\kappa_{\y}$. Thus, AGA requires $\tilde{O}(\sqrt{\kappa_{\x}\kappa_{\y}})$ iterations in the outer loop to solve the maximization of $\Psi$, which gives a total gradient complexity $\tilde{O}(\kappa_{\x}\sqrt{\kappa_{\y}})$.

The key observation here is that although Algorithm \ref{alg:Proxy} is slow for general strongly-convex-strongly-concave functions, the functions $\tilde{g}$ of the form \eqref{eq:prox_minimax} in the proximal steps have a crucial property that $\kappa_\x = O(1)$ if the proximal parameter $\ell$ is chosen to be the smoothness parameter of function $f$. Therefore, when $f(\x, \cdot)$ is strongly concave, by Theorem \ref{Theorem:maximin-AG2}, it only takes Algorithm \ref{alg:Proxy} $\tilde{O}(\sqrt{\kappa_{\y}})$ gradient evaluations to solve the proximal subproblem \eqref{eq:prox_minimax}, which is very efficient. We will see the consequences of this fact in the following section.


\section{Accelerating Convex-Concave Optimization}\label{sec:result_convex_concave}
In this section, we present our main results for accelerating convex-concave optimization. We first present our new near-optimal algorithm and its theoretical guarantee for optimizing strongly-convex-strongly-concave functions. Then, we use simple reduction arguments to obtain results for strongly-convex-concave and convex-concave functions.

\subsection{Strongly-convex-strongly-concave setting}
With the algorithm components from Section \ref{sec:component} in hand, we are now ready to state our near-optimal algorithm. Algorithm \ref{alg:Minimax_APG} is a simple combination of Algorithm \ref{alg:APPA} and Algorithm \ref{alg:Proxy}. Its outer loop performs an inexact APPA to minimize the function $\Phi(\cdot) := \max_{\y \in \YCal} f(\cdot, \y)$, while the inner loop uses \AGPROX~to solve the proximal subproblem \eqref{eq:prox_minimax_original_form}, which is equivalent to solving \eqref{eq:prox_minimax}. At the end, after finding a near-optimal $\x_T$, Algorithm \ref{alg:Minimax_APG} performs another AGD on the function $-f(\x_T, \cdot)$ to find a near-optimal $\y_T$. The theoretical guarantee for the algorithm is given in the following theorem.
\begin{theorem} \label{thm:scsc}
Assume that $f$ is $\ell$-smooth and $\mu_\x$-strongly-convex-$\mu_\y$-strongly-concave. Then there exists $T > 0$ such that the output $(\hat{\x}, \hat{\y}) = \textsc{\AGCC}(f, \x_0, \y_0, \ell, \mu_\x, \mu_\y, \epsilon, T)$ is an $\epsilon$-saddle point, and the total number of gradient evaluations is bounded by
\begin{equation*}
O\left(\sqrt{\kappa_\x \kappa_\y} \log^3 \left(\frac{(\kappax+\kappay) \ell (D_\x^2 + D_\y^2)}{\epsilon}\right)\right), 
\end{equation*}
where $\kappa_\x = \ell/\mu_\x$ and $\kappa_\y = \ell/\mu_\y$ are condition numbers.
\end{theorem}
Theorem \ref{thm:scsc} asserts that Algorithm \ref{alg:Minimax_APG} finds $\epsilon$-saddle points in $\tilde{O}(\sqrt{\kappa_\x \kappa_\y})$ gradient evaluations, matching the lower bound~\citep{Ibrahim-2019-Lower, Zhang-2019-Lower}, up to logarithmic factors. At a high level, despite the function $\Phi$ having undesirable smoothness properties, APPA minimizes $\Phi$ in the outer loop using $\tilde{O}(\sqrt{\kappa_\x})$ iterations according to Theorem \ref{Theorem:inexact-APPA}, regardless of the smoothness of $\Phi$. According to the discussion in Section \ref{sec:subsolver}, \AGPROX~solves the proximal step in the inner loop using $\tilde{O}(\sqrt{\kappa_\y})$ gradient evaluations, since the condition number of $g_t(\cdot, \y)$ for any $\y \in \YCal$ is $O(1)$. This gives the total gradient complexity $\tilde{O}(\sqrt{\kappa_\x\kappa_\y})$.
\begin{algorithm}[!t]
\caption{$\textsc{\AGCC}(f, \x_0, \y_0, \ell, \mu_\x, \mu_\y, \epsilon, T)$}\label{alg:Minimax_APG}
\begin{algorithmic}[1]
\STATE \textbf{Input:} initial point $\x_0, \y_0$, proximity $\ell$, strongly-convex parameter $\mu$, tolerance $\delta$, iteration $T$. 
\STATE \textbf{Initialize:} $\tilde{\x}_0 \leftarrow \x_0$, $\kappax \leftarrow \frac{\ell}{\mu_\x}$, $\theta \leftarrow \frac{2\sqrt{\kappax}-1}{2\sqrt{\kappax}+1}$, $\delta \leftarrow \frac{\epsilon}{(10\kappax\kappay)^4}$ and $\tilde{\epsilon} \leftarrow \frac{\epsilon}{10^2\kappax\kappay}$. 
\FOR{$t = 1, \cdots, T$}
\STATE denote $g_t(\cdot, \cdot)$ where $g_t(\x, \y) := f(\x, \y) + \ell \norm{\x - \tilde{\x}_{t-1}}^2$. \label{line:Minimax_prox}
\STATE $\x_t \leftarrow \textsc{\AGPROX}(g_t, \x_0, \y_0, 3\ell, 2\ell, \mu_\y, \delta)$
\STATE $\tilde{\x}_{t} \leftarrow \x_{t} + \theta(\x_{t} - \x_{t-1})$. 
\ENDFOR
\STATE $\tilde{\y} \leftarrow \textsc{AGD}(-f(\x_T, \cdot), \y_0, \ell, \mu_{\y}, \tilde{\epsilon})$.
\STATE $\y_T \leftarrow \proj_\YCal\left(\tilde{\y} + (1/2\kappa_\x\ell)\grady f(\x_T, \tilde{\y})\right)$. 
\STATE \textbf{Output:} $(\x_T, \y_T)$. 
\end{algorithmic}
\end{algorithm}

\subsection{Strongly-convex-concave setting}
Our result in the strongly-convex-strongly-concave setting readily implies a near-optimal result in the strongly-convex-concave setting. Consider the following auxiliary function for an arbitrary $\y_0 \in \YCal$ which is defined by
\begin{equation}\label{eq:reduction1}
f_{\epsilon, \y} (\x, \y) \ := \ f(\x, \y) - (\epsilon/4D^2_\y)\norm{\y - \y_0}^2. 
\end{equation}
By construction, it is clear that the difference between $f$ and $f_{\epsilon, \y}$ is small in terms of function value:
\begin{equation*}
\max_{(\x, \y) \in \XCal \times \YCal} |f(\x, \y) - f_{\epsilon, \y} (\x, \y)| \ \leq \ \epsilon/4. 
\end{equation*}
This implies, according to Definition \ref{Def:eps_saddle}, that any $(\epsilon/2)$-saddle point of function $f_{\epsilon, \y}$ is also a $\epsilon$-saddle point of function $f$, and thus it is sufficient to only solve the problem $\min_{\x\in \XCal}\max_{\x\in \YCal} f_{\epsilon, \y}(\x, \y)$. Finally, when $f$ is a $\mu_\x$-strongly-convex-concave function, $f_{\epsilon, \y}$ becomes $\mu_\x$-strongly-convex-$\epsilon/(2D_\y^2)$-strongly-concave, which can be fed into Algorithm \ref{alg:Minimax_APG} to obtain the following result.
\begin{corollary} \label{cor:scc}
Assume that $f$ is $\ell$-smooth and $\mu_\x$-strongly-convex-concave. Then there exists $T>0$ such that the output $(\hat{\x}, \hat{\y}) = \textsc{\AGCC}(f_{\epsilon, \y}, \x_0, \y_0, \ell, \mu_\x, \epsilon/(4D^2_{\y}), \epsilon/2, T)$ is an $\epsilon$-saddle point, and the total number of gradient evaluations is bounded by
\begin{equation*}
O\left(\sqrt{\frac{\kappa_\x \ell}{\epsilon}}\Dy \log^3 \left(\frac{\kappax \ell (D_\x^2 + D_\y^2)}{\epsilon}\right)\right)
\end{equation*}
where $\kappa_\x = \ell/\mu_\x$ is the condition number, and $f_{\epsilon, \y}$ is defined as in \eqref{eq:reduction1}.
\end{corollary}

\subsection{Convex-concave setting}
Similar to the previous subsection, when $f$ is only convex-concave,  we can construct following strongly-convex-strongly-concave function $f_\epsilon$:
\begin{equation}\label{eq:reduction2}
f_{\epsilon}(\x, \y) = f(\x, \y) + (\epsilon/8\Dx^2)\norm{\x - \x_0}^2 - (\epsilon/8\Dy^2)\norm{\y - \y_0}^2,
\end{equation}
which can be fed into Algorithm \ref{alg:Minimax_APG} to obtain the following result.
\begin{corollary} \label{cor:cc}
Assume function $f$ is $\ell$-smooth and convex-concave, then there exists $T>0$, where the output $(\hat{\x}, \hat{\y}) = \textsc{\AGCC}(f_{\epsilon}, \x_0, \y_0, \ell, \epsilon/(4\Dx^2), \epsilon/(4\Dy^2), \epsilon/2, T)$ will be an $\epsilon$-saddle point, and the total number of gradient evaluations is bounded by
\begin{equation*}
O\left(\frac{\ell\Dx\Dy}{\epsilon} \log^3 \left(\frac{\ell (D_\x^2 + D_\y^2)}{\epsilon}\right)\right),
\end{equation*}
where $f_{\epsilon}$ is defined as in \eqref{eq:reduction2}.
\end{corollary}


\section{Accelerating Nonconvex-Concave Optimization} \label{sec:result_nonconvex_concave}
In this section, we present methods for accelerating nonconvex-concave optimization. Similar to Section \ref{sec:result_convex_concave}, we first present our algorithm and its theoretical guarantee for optimizing nonconvex-strongly-concave functions. We then use a simple reduction argument to obtain results for nonconvex-concave functions. This section present results using the stationarity of the function $f$ (Definition \ref{Def:eps_stationary}) as an optimality measure. Please see Appendix~\ref{sec:app-nonconvex} for additional results using the stationarity of the function $\Phi(\cdot):=\max_{y\in \YCal} f(\cdot, \y)$ as the optimality measure (Definition \ref{Def:NSC-Stationary} and \ref{Def:NC-Stationary}).

\subsection{Nonconvex-strongly-concave setting}
Our algorithm for nonconvex-strongly-concave optimization is described in Algorithm \ref{alg:Minimax_PG}. Similar to Algorithm \ref{alg:Minimax_APG}, we still use our accelerated solver \AGPROX~for the same proximal subproblem in the inner loop. The only minor difference is that, in the outer loop, Algorithm \ref{alg:Minimax_PG} only uses the Proximal Point Algorithm (PPA) on function $\Phi(\cdot):= \max_{y\in \YCal}f(\cdot, \y)$ without acceleration (or momentum steps). This is due to fact that gradient descent is already optimal among all first-order algorithm for finding stationary points of smooth nonconvex functions~\citep{Carmon-2019-LowerI}. The standard acceleration technique will not help for smooth nonconvex functions. We presents the theoretical guarantees for Algorithm \ref{alg:Minimax_PG} in the following theorem.
\begin{theorem}\label{thm:nsc}
Assume that $f$ is $\ell$-smooth and $f(\x, \cdot)$ is $\mu_\y$-strongly-concave for all $\x$.  Then there exists $T > 0$ such that the output $(\hat{\x}, \hat{\y}) = \textsc{\AGNC}(f, \x_0, \y_0, \ell, \mu_{\y}, \epsilon, T)$ is an $\epsilon$-stationary point of $f$ with probability at least $2/3$, and the total number of gradient evaluations is bounded by
\begin{equation*}
O\left(\frac{\ell \Delta_\Phi}{\epsilon^2}\cdot \sqrt{\kappa_\y} \log^2 \left(\frac{\kappa_\y \ell(\tilde{D}_\x^2 + \Dy^2)}{\epsilon}\right)\right),
\end{equation*}
where $\kappa_\y = \ell/\mu_\y$ is the condition number, $\Delta_\Phi = \Phi(\x_0) - \min_{\x \in \br^m} \Phi(\x)$ is the initial function value gap and $\tilde{D}_\x = \|\x_0 - \x_{g_1}^\star(\y_0)\|$ is the initial distance where $\x_g^\star(\y_0) = \argmin_{\x \in \XCal} g(\x, \y_0)$. 
\end{theorem}
Theorem \ref{thm:nsc} claims that Algorithm \ref{alg:Minimax_PG} will find an $\epsilon$-stationary point, with at least constant probability, in $\tilde{O}(\sqrt{\kappa_\y}/\epsilon^2)$ gradient evaluations. Similar to Theorem \ref{thm:scsc}, the inner loop takes $\tilde{O}(\sqrt{\kappa_\y})$ gradient evaluations to solve the proximal step since the condition number of $g_t(\cdot, \y)$ is $O(1)$ for any $\y \in \YCal$. In the outer loop, regardless of the smoothness of $\Phi(\cdot)$, PPA with proximal parameter $\ell$ is capable of finding the stationary point in $\tilde{O}(1/\epsilon^2)$ iterations. In total, the gradient complexity is $\tilde{O}(\sqrt{\kappa_\y}/\epsilon^2)$.
\begin{algorithm}[t]
\caption{$\textsc{\AGNC}(g, \x_0, \y_0, \ell, \mu_\y, \epsilon, T)$}\label{alg:Minimax_PG}
\begin{algorithmic}[1]
\STATE \textbf{Input:} initial point $\x_0, \y_0$, proximity $\ell$, strongly-convex parameter $\mu$, tolerance $\delta$, iteration $T$. 
\STATE \textbf{Initialize:} set $\delta \leftarrow \frac{\epsilon^2}{(10\kappay)^4\ell} \cdot (\frac{\epsilon}{\ell D_\y})^2$. 
\FOR{$t = 1, \cdots, T$}
\STATE denote $g_t(\cdot, \cdot)$ where $g_t(\x, \y) := f(\x, \y) + \ell \norm{\x - \x_{t-1}}^2$.
\STATE $\x_t \leftarrow \textsc{\AGPROX}(g_t, \x_0, \y_0, 3\ell, \ell, \mu, \delta)$. 
\ENDFOR
\STATE sample $s$ uniformly from $\{1, 2, \cdots, T\}$. 
\STATE $\y_{s} \leftarrow \textsc{AGD}(-f(\x_s, \cdot), \y_0, \ell, \mu, \delta)$.
\STATE \textbf{Output:} $(\x_s, \y_s)$. 
\end{algorithmic}
\end{algorithm}
\subsection{Nonconvex-concave setting}
Our result in the nonconvex-strongly-concave setting readily implies a fast result in the nonconvex-concave setting. Consider the following auxiliary function for an arbitrary $\y_0 \in \YCal$:
\begin{equation}\label{eq:reduction3}
\tilde{f}_{\epsilon}(\x, \y) = f(\x, \y) - (\epsilon/4D_\y) \norm{\y - \y_0}^2.
\end{equation}
By construction, it is clear that the gradient of $f$ and $\tilde{f}_{\epsilon}$ are close in the sense 
$$\max_{(\x, \y) \in \mathbb{R}^m \times \YCal} \norm{\grad f(\x, \y) - \grad \tilde{f}_{\epsilon} (\x, \y)} \le \epsilon/4.$$ 
This implies that any $(\epsilon/2)$-stationary point of $\tilde{f}_{\epsilon}$ is also a $\epsilon$-stationary point of $f$, and thus it is sufficient to solve the problem $\min_{\x\in \XCal}\max_{\x\in \YCal} \tilde{f}_{\epsilon}(\x, \y)$. Finally, the function $\tilde{f}_{\epsilon}(\x, \cdot)$ is always $\epsilon/(2D_\y)$-strongly-concave, which can be fed into Algorithm \ref{alg:Minimax_PG} to obtain the following result.
\begin{corollary}\label{cor:nc}
Assume that $f$ is $\ell$-smooth and $f(\x, \cdot)$ is concave for all $\x$.  Then there exists $T>0$ such that the output $(\hat{\x}, \hat{\y}) = \textsc{\AGNC}(\tilde{f}_{\epsilon}, \x_0, \y_0, \ell, \epsilon/(2D_{\y}), \epsilon/2, T)$ is an $\epsilon$-stationary point of $f$ with probability at least $2/3$, and the total number of gradient evaluations is bounded by
\begin{equation*}
O\left(\frac{\ell\Delta_\Phi}{\epsilon^{2}} \cdot \sqrt{\frac{\ell \Dy}{\epsilon}}\log^2 \left(\frac{\ell(\tilde{D}_\x^2 + \Dy^2)}{\epsilon}\right)\right),
\end{equation*}
where $\Dy > 0$, $\Delta_\Phi = \Phi(\x_0) - \min_{\x \in \br^m} \Phi(\x)$ is the initial function value gap and $\tilde{D}_\x = \|\x_0 - \x_{g_1}^\star(\y_0)\|$ is the initial distance where $\x_g^\star(\y_0) = \argmin_{\x \in \XCal} g(\x, \y_0)$. 
\end{corollary}

\section{Conclusions}\label{sec:conclusion}
This paper has provided the first set of \textit{near-optimal} algorithms for strongly-convex-(strongly)-concave minimax optimization problems and the state-of-the-art algorithms for nonconvex-(strongly)-concave minimax optimization problems. For the former class of problems, our algorithms match the lower complexity bound for first-order algorithms~\citep{Ouyang-2019-Lower, Ibrahim-2019-Lower, Zhang-2019-Lower} up to logarithmic factors. For the latter class of problems, our algorithms achieve the best known upper bound. In the future research, one important direction is to investigate the lower complexity bound of first-order algorithms for nonconvex-(strongly)-concave minimax problems. Despite several striking results on lower complexity bounds for nonconvex smooth problems~\citep{Carmon-2019-LowerI, Carmon-2019-LowerII}, this problem remains challenging as solving it requires a new construction of ``chain-style" functions and resisting oracles.

\section*{Acknowledgments}
We would like to thank three anonymous referees for constructive suggestions that improve the quality of this paper. This work was supported in part by the Mathematical Data Science program of the Office of Naval Research under grant number N00014-18-1-2764.
\bibliographystyle{plainnat}
\bibliography{ref}

\appendix

\section{Additional Results for Nonconvex-Concave Optimization}\label{sec:app-nonconvex}
In this section, we present our results for nonconvex-concave optimization using stationary of $\Phi(\cdot) := \max_{\y\in \YCal}f(\cdot, \y)$ (Definition \ref{Def:NSC-Stationary} and Definition \ref{Def:NC-Stationary}) as the optimality measure.

\subsection{Optimality notion based on Moreau envelope}\label{sec:envelope}
We present another optimality notion based on Moreau envelope for nonconvex-concave setting in which $f(\cdot, \y)$ is not necessarily convex for each $\y \in \YCal$ but $f(\x, \cdot)$ is concave for each $\x \in \XCal$. For simplicity, we let $\XCal = \br^m$ and $\YCal$ be convex and bounded. In general, finding a global saddle point of $f$ is intractable since solving the special case with a singleton $\YCal$ globally is already NP-hard~\citep{Murty-1987-Some} as mentioned in the main text. 

One approach, inspired by nonconvex optimization, is to equivalently reformulate problem~\eqref{prob:minimax} as the following nonconvex minimization problem:
\begin{equation}\label{prob:alternative}
\min_{\x \in \br^m} \ \left\{\Phi(\x) := \max_{\y \in \YCal} f(\x, \y) \right\},    
\end{equation}
and define an optimality notion for the local surrogate of global optimum of $\Phi$. In robust learning, $\x$ is the classifier while $\y$ is the adversarial noise. Practitioners are often only interested in finding a robust classifier $\x$ instead of an adversarial response $\y$ to each data point. Such a stationary point $\x$ precisely corresponds to a robust classifier that is stationary to the robust classification error.

If $f(\x, \cdot)$ is further assumed to be strongly concave for each $\x \in \br^m$, then $\Phi$ is smooth and a standard optimality notion is the stationary point. 
\begin{definition}\label{Def:NSC-Stationary}
We call $\hat{\x}$ an $\epsilon$-stationary point of a \textsf{smooth} function $\Phi$ if $\left\|\grad\Phi(\hat{\x})\right\| \leq \epsilon$. If $\epsilon = 0$, then $\hat{\x}$ is called a stationary point. 
\end{definition}
In contrast, when $f(\x, \cdot)$ is merely concave for each $\x \in \XCal$, $\Phi$ is not necessarily smooth and even not differentiable. A weaker sufficient condition for the purpose of our paper is the weak convexity.
\begin{definition}\label{Def:NC-Weak-Convex}
A function $\Phi: \br^d \rightarrow \br$ is $L$-weakly convex if $\Phi(\cdot) + (L/2)\left\|\cdot\right\|^2$ is convex.
\end{definition}
First, a function $\Phi$ is $\ell$-weakly convex if it is $\ell$-smooth. Second, the subdifferential of a $\ell$-weakly convex function $\Phi$ can be uniquely determined by the subdifferential of $\Phi(\cdot) + (\ell/2)\|\cdot\|^2$. This implies that the optimality notion can be defined by a point $\x \in \br^m$ with at least one small subgradient: $\min_{\xi \in \partial \Phi(\x)} \|\xi\| \leq \epsilon$. Unfortunately, this notion can be restrictive if $\Phi$ is nonsmooth. Considering a one-dimensional function $\Phi(\cdot) = |\cdot|$, a point $\x$ must be $0$ if it satisfies the optimality notion with $\epsilon \in [0, 1)$. This means that finding a sufficiently accurate solution under such optimality notion is as difficult as solving the minimization exactly. Another popular optimality notion is based on the \textit{Moreau envelope} of $\Phi$ when $\Phi$ is weakly convex~\citep{Davis-2019-Stochastic}.
\begin{definition}
A function $\Phi_\lambda$ is the Moreau envelope of $\Phi$ with $\lambda > 0$ if for $\forall \x \in \br^m$, that 
\begin{equation*}
\Phi_\lambda(\x) \ = \ \min_{\w \in \br^m} \Phi(\w) + (1/2\lambda)\|\w - \x\|^2. 
\end{equation*}
\end{definition}
\begin{lemma}[Properties of Moreau envelopes]\label{Lemma:NC-Moreau-Envelope}
If the function $\Phi(\cdot)$ is $\ell$-weakly convex, its Moreau envelope $\Phi_{1/2\ell}(\cdot)$ is $4\ell$-smooth with the gradient $\grad\Phi_{1/2\ell}(\cdot) = 2\ell(\cdot - \prox_{\Phi/2\ell}(\cdot))$ in which a point $\prox_{\Phi/2\ell}(\cdot) = \argmin_{\w \in \br^m} \{\Phi(\w) + \ell\|\w - \cdot\|^2\}$ is defined.  
\end{lemma}
Thus, an $\epsilon$-stationary point of an $\ell$-weakly convex function $\Phi$ can be alternatively defined as a point $\hat{\x}$ satisfying that the gradient norm of Moreau envelope $\|\nabla \Phi_{1/2\ell}(\hat{\x})\|$ is small. 
\begin{definition}\label{Def:NC-Stationary}
We call $\hat{\x}$ an $\epsilon$-stationary point of a $\ell$-\textsf{weakly convex} function $\Phi$ if $\left\|\grad\Phi_{1/2\ell}(\hat{\x})\right\| \leq \epsilon$. If $\epsilon = 0$, then $\hat{\x}$ is called a stationary point.  
\end{definition}
\begin{lemma}[Properties of $\epsilon$-stationary point]\label{Lemma:NC-Stationary}
If $\hat{\x}$ is an $\epsilon$-stationary point of a $\ell$-weakly convex function $\Phi$, then there exists $\bar{\x} \in \br^m$ such that $\min_{\xi \in \partial \Phi(\bar{\x})} \norm{\xi} \leq \epsilon$ and $\norm{\hat{\x} - \bar{\x}} \leq \epsilon/2\ell$.  
\end{lemma}
Lemma~\ref{Lemma:NC-Stationary} shows that an $\epsilon$-stationary point defined by the Moreau envelope can be interpreted as the relaxation for a point with at least one small subgradient. In particular, if $\hat{\x}$ is an $\epsilon$-stationary point of a $\ell$-weakly convex function $\Phi$, then it is close to a point which has small subgradient.


\subsection{Nonconvex-strongly-concave setting}
In the setting of nonconvex-strongly-concave function, we still use Algorithm \ref{alg:Minimax_PG}. Similar to Theorem \ref{thm:nsc}, we can obtain a guarantee, which finds a point $\hat{\x}$ satisfying $\norm{\grad \Phi(\hat{\x})} \le \epsilon$ in the same number of iterations as in Theorem \ref{thm:nsc}.
\begin{theorem}\label{thm:nsc-Moreau}
Assume that $f$ is $\ell$-smooth and $f(\x, \cdot)$ is $\mu_\y$-strongly-concave for all $\x$. Then there exists $T > 0$ such that the output $(\hat{\x}, \hat{\y}) = \textsc{\AGNC}(f, \x_0, \y_0, \ell, \mu_{\y}, \epsilon, T)$ satisfies $\norm{\grad \Phi(\hat{\x})} \leq \epsilon$ with probability at least $2/3$, and the total number of gradient evaluations is bounded by
\begin{equation*}
O\left(\frac{\ell \Delta_\Phi}{\epsilon^2}\cdot \sqrt{\kappa_\y} \log^2 \left(\frac{\kappa_\y \ell(\tilde{D}_\x^2 + \Dy^2)}{\epsilon}\right)\right)
\end{equation*}
where $\kappa_\y = \ell/\mu_\y$ is the condition number, $\Delta_\Phi = \Phi(\x_0) - \min_{\x \in \br^m} \Phi(\x)$ is the initial function value gap and $\tilde{D}_\x = \|\x_0 - \x_{g_1}^\star(\y_0)\|$ is the initial distance where $\x_g^\star(\y_0) = \argmin_{\x \in \XCal} g(\x, \y_0)$. 
\end{theorem}
\subsection{Nonconvex-concave setting}
We can similarly reduce the problem of optimizing a nonconvex-concave function to the problem of optimizing a nonconvex-strongly-concave function. The only caveat is that, in order to achieve the near-optimal point using Definition \ref{Def:NC-Stationary} as optimality measure, we can only add a $O(\epsilon^2)$ term as follows:
\begin{equation}\label{eq:reduction4}
\bar{f}_{\epsilon}(\x, \y) = f(\x, \y) - (\epsilon^2/200\ell\Dy^2)\norm{\y - \y_0}^2. 
\end{equation}
Now $\bar{f}_{\epsilon}(\x, \cdot)$ is only $\epsilon^2/(100\ell D^2_\y)$-concave, by feeding it to Algorithm \ref{alg:Minimax_PG} and through a slightly more complicated reduction argument, we can only obtain gradient complexity bound of $\tilde{O}(\epsilon^{-3})$ instead of $\tilde{O}(\epsilon^{-2.5})$ as in Corollary \ref{cor:nc}. Formally, we have
\begin{corollary}\label{cor:nc-Moreau}
Assume that $f$ is $\ell$-smooth, and $f(\x, \cdot)$ is concave for all $\x$. Then there exists $T>0$ such that the output $(\hat{\x}, \hat{\y}) = \textsc{\AGNC}(\bar{f}_{\epsilon}, \x_0, \y_0, \ell, \epsilon^2/(100\ell D^2_{\y}), \epsilon/10, T)$ satisfies $\norm{\grad \Phi_{1/2\ell}(\hat{\x})} \leq \epsilon$ with probability at least $2/3$, and the total number of gradient evaluations is bounded by
\begin{equation*}
\cO\left(\frac{\ell^2\Dy\Delta_\Phi}{\epsilon^3} \log^2 \left(\frac{\ell(\tilde{D}_\x^2 + \Dy^2)}{\epsilon}\right)\right)
\end{equation*}
where $\Dy > 0$, $\Delta_\Phi = \Phi(\x_0) - \min_{\x \in \br^m} \Phi(\x)$ is the initial function value gap and $\tilde{D}_\x = \|\x_0 - \x_{g_1}^\star(\y_0)\|$ is the initial distance where $\x_g^\star(\y_0) = \argmin_{\x \in \XCal} g(\x, \y_0)$. 
\end{corollary}


\section{Proofs for Algorithm Components}
In this section, we present proofs for our algorithm components.

\subsection{Proof of Theorem~\ref{Theorem:AGD}}
We divide the proof into three parts. In the first part, we show that the output $\hat{\x}$ satisfies $g(\hat{\x}) \leq \min_{\x \in \XCal} g(\x) + \epsilon$. In the second part, we derive the sufficient condition for guaranteeing the stopping criteria in Algorithm~\ref{alg:AGD}. In the third part, we derive the gradient complexity of the algorithm using the condition derived in the second part.

\paragraph{Part I.} Let $\tilde{\x}_t = \PX(\x_t - (1/\ell)\grad g(\x_t))$ be defined as the point achieved by one-step projected gradient descent from $\x_t$. Since $g$ is $\ell$-smooth and $\mu$-strongly convex, it is straightforward to derive from~\citet[Corollary~2.3.2]{Nesterov-2018-Lectures} that
\begin{equation*}
g(\x) \ \geq \ g(\tilde{\x}_t) + \ell(\x_t - \tilde{\x}_t)^\top(\x - \x_t) + \frac{\ell}{2}\|\x_t - \tilde{\x}_t\|^2 + \frac{\mu}{2}\|\x - \x_t\|^2, \quad \textnormal{for all } \x \in \XCal. 
\end{equation*}
Using the Young's inequality, we have $(\x_t - \tilde{\x}_t)^\top(\x - \x_t) \geq -(1/2)(\|\x_t - \tilde{\x}_t\|^2 + \|\x - \x_t\|^2)$. Putting these pieces together with $\x = \x^\star$ yields that
\begin{equation*}
g(\tilde{\x}_t) - \min_{\x \in \XCal} g(\x) \ = \ g(\tilde{\x}_t) - g(\x^\star) \ \leq \ \left(\frac{\ell - \mu}{2}\right)\|\x_t - \x^\star\|^2. 
\end{equation*}
Without loss of generality, we assume $\ell > \mu$. Indeed, if $\ell = \mu$, then one-step projected gradient descent from any points in $\XCal$ guarantees that $g(\tilde{\x}_t) - \min_{\x \in \XCal} g(\x) = 0$. Since $\hat{\x} = \tilde{\x}_t$ in Algorithm~\ref{alg:AGD}, it suffices to show that the following statement holds true, 
\begin{equation}\label{criterion:AGD-update}
\|\x_t - \PX(\x_t - (1/\ell)\grad g(\x_t))\| \ \leq \ \sqrt{\frac{\epsilon}{2\kappa^2(\ell - \mu)}} \ \Longrightarrow \ \|\x_t - \x^\star\| \ \leq \ \sqrt{\frac{2\epsilon}{\ell - \mu}}.  
\end{equation}
Let $\tilde{\x}_t = \PX(\x_t - (1/\ell)\grad g(\x_t))$ be defined as the point achieved by one-step projected gradient descent from $\x_t$, the $\ell$-smoothness of $g$ implies
\begin{equation}\label{prop:AGD-inequality-first-app}
\|\tilde{\x}_t - \x^\star\| \ \leq \ \|\x_k - \x^\star\|. 
\end{equation}
Using the definition of $\tilde{\x}_t$ and $\x^\star$, we have
\begin{equation*}
(\x^\star - \tilde{\x}_t)^\top(\ell(\tilde{\x}_t - \x_t) + \grad g(\x_t)) \ \geq \ 0, \qquad (\tilde{\x}_t - \x^\star)^\top\grad g(\x^\star) \ \geq \ 0. 
\end{equation*}
Summing up the above two inequalities and rearranging yields that 
\begin{equation*}
(\x^\star - \x_t)^\top(\grad g(\x_t) - \grad g(\x^\star)) \ \geq \ \ell(\x^\star - \tilde{\x}_t)^\top(\x_t - \tilde{\x}_t) + (\tilde{\x}_t - \x_t)^\top(\grad g(\x_t) - \grad g(\x^\star)). 
\end{equation*}
Since $g$ is $\ell$-smooth and $\mu$-strongly convex, we have
\begin{equation*}
- \mu\|\x_t - \x^\star\|^2 \ \geq \ - \ell\|\x_t - \tilde{\x}_t\|\left(\|\x^\star - \tilde{\x}_t\| + \|\x^\star - \x_t\|\right) \ \overset{~\eqref{prop:AGD-inequality-first-app}}{\geq} \ - 2\ell\|\x_t - \tilde{\x}_t\|\|\x_t - \x^\star\|. 
\end{equation*}
Therefore, we conclude that 
\begin{equation*}
\|\x_t - \x^\star\| \ \leq \ 2\kappa\|\x_t - \tilde{\x}_t\| \ = \ 2\kappa\|\x_t - \PX(\x_t - (1/\ell)\grad g(\x_t))\| \ \overset{~\eqref{criterion:AGD-update}}{\leq} \ \sqrt{\frac{2\epsilon}{\ell - \mu}}. 
\end{equation*}
\paragraph{Part II.} We first show that 
\begin{equation*}
\|\x_t - \x^\star\| \ \leq \ \frac{1}{3\kappa}\sqrt{\frac{\epsilon}{2(\ell - \mu)}} \ \Longrightarrow \ \|\x_t - \PX(\x_t - (1/\ell)\grad g(\x_t))\| \ \leq \ \sqrt{\frac{\epsilon}{2\kappa^2(\ell - \mu)}}. 
\end{equation*}
By the definition of $\x^\star$, we have $\x^\star = \PX(\x^\star - (1/\ell)\grad g(\x^\star))$. This equation together with the triangle inequality and the nonexpansiveness of $\PX$ yields that $\|\x_t - \PX(\x_t - (1/\ell)\grad g(\x_t))\| \leq 3\|\x_t - \x^\star\|$ which implies the desired result. Then we derive a sufficient condition for guaranteeing that $\|\x_t - \x^\star\| \leq (1/(3\kappa))\sqrt{\epsilon/(2(\ell - \mu))}$. Since $g$ is $\mu$-strongly convex and $\x_t \in \XCal$,~\citet[Theorem~2.1.5]{Nesterov-2018-Lectures} together with the fact that $(\x_t - \x^\star)^\top\grad g(\x^\star) \geq 0$ implies that 
\begin{equation*}
\|\x_t - \x^\star\|^2 \ \leq \ \frac{2}{\mu}\left(g(\x_t) - \min_{\x \in \XCal} g(\x)\right).  
\end{equation*}
Putting these pieces together yields the desired sufficient condition as follows,   
\begin{equation}\label{prop:AGD-inequality-second-app}
g(\x_t) - \min_{\x \in \XCal} g(\x) \ \leq \ \frac{\epsilon}{36\kappa^3}. 
\end{equation}
\paragraph{Part III.} We proceed to derive the gradient complexity of the algorithm using the condition in Eq.~\eqref{prop:AGD-inequality-second-app}. Since Algorithm~\ref{alg:AGD} is exactly Nesterov's accelerated gradient descent, standard arguments based on estimate sequence~\citep{Nesterov-2018-Lectures} implies
\begin{equation*}
g(\x_t) - \min_{\x \in \XCal} g(\x) \ \leq \ \left(1 - \frac{1}{\sqrt{\kappa}}\right)^t\left(g(\x_0) - \min_{\x \in \XCal} g(\x) + \frac{\mu\|\x^\star - \x_0\|^2}{2}\right). 
\end{equation*}
Therefore, the gradient complexity of Algorithm~\ref{alg:AGD} to guarantee Eq.~\eqref{prop:AGD-inequality-second-app} is bounded by 
\begin{equation*}
O\left(1 + \sqrt{\kappa}\log\left(\frac{\kappa^3\ell\|\x_0 - \x^\star\|^2}{\epsilon}\right)\right). 
\end{equation*}
This completes the proof. 

\subsection{Proof of Theorem~\ref{Theorem:inexact-APPA}}
Letting $\hat{\x} = \textsc{Inexact-APPA}(g, \x_0, \ell, \mu, \epsilon, T)$. Since $\hat{\x} = \x_T$, it suffices for us to estimate an lower bound for the maximum number of iterations $T$ such that $g(\x_T) \leq \min_{\x \in \XCal} g(\x) + \epsilon$. The following technical lemma is crucial to the subsequent analysis. 
\begin{lemma}\label{Lemma:inexact-APPA}
For any $\x \in \XCal$ and $\{(\x_t, \tilde{\x}_t)\}_{t \geq 0}$ generated by Algorithm~\ref{alg:APPA}, we have
\begin{equation}\label{inequality-inexact-APPA}
g(\x) \ \geq \ g(\x_t) - 2\ell(\x - \tilde{\x}_{t-1})^\top(\x_t - \tilde{\x}_{t-1}) + 2\ell\|\x_t - \tilde{\x}_{t-1}\|^2 + \frac{\mu\|\x - \x_t\|^2}{4} - 7\kappa\delta. 
\end{equation}
\end{lemma}
\begin{proof}
Using the definition of $\x_t$ in Algorithm~\ref{alg:APPA}, we have
\begin{equation*}
g(\x_t) + \ell\|\x_t - \tilde{\x}_{t-1}\|^2 \ \leq \ \min_{\x \in \XCal} \left\{g(\x) + \ell\|\x - \tilde{\x}_{t-1}\|^2\right\} + \delta. 
\end{equation*}
Defining $\x_t^\star = \argmin_{\x \in \XCal} \{g(\x) + \ell\|\x - \tilde{\x}_{t-1}\|^2\}$ and using $\mu$-strongly convexity of $g$, we have the following for any $\x \in \XCal$: 
\begin{equation*}
g(\x) \ \geq \ g(\x_t^\star) + \ell\|\x_t^\star - \tilde{\x}_{t-1}\|^2 - \ell\|\x - \tilde{\x}_{t-1}\|^2 + \left(\ell + \frac{\mu}{2}\right)\|\x - \x_t^\star\|^2. 
\end{equation*}
Equivalently, we have
\begin{eqnarray*}
g(\x) & \geq & g(\x_t) + \ell\|\x_t - \tilde{\x}_{t-1}\|^2 - \ell\|\x - \tilde{\x}_{t-1}\|^2 + \left(\ell + \frac{\mu}{2}\right)\|\x - \x_t^\star\|^2 - \delta \\
& \geq & g(\x_t) - 2\ell(\x - \x_t)^\top(\x_t - \tilde{\x}_{t-1}) - \ell\|\x - \x_t\|^2 + \left(\ell + \frac{\mu}{2}\right)\|\x - \x_t^\star\|^2 - \delta. 
\end{eqnarray*}
On the other hand, we have
\begin{equation*}
\left(\ell + \frac{\mu}{2}\right)\|\x - \x_t^\star\|^2 - \ell\|\x - \x_t\|^2 \ = \ \frac{\mu\|\x - \x_t\|^2}{2} + (2\ell + \mu)(\x - \x_t)^\top(\x_t - \x_t^\star) + \left(\ell + \frac{\mu}{2}\right)\|\x_t - \x_t^\star\|^2
\end{equation*}
Using Young's inequality yields 
\begin{equation*}
(\x - \x_t)^\top(\x_t - \x_t^\star) \ \geq \ -\frac{\mu\|\x - \x_t\|^2}{4(2\ell+\mu)} - (1+2\kappa)\|\x_t - \x_t^\star\|^2. 
\end{equation*}
Putting these pieces together yields that 
\begin{equation*}
g(\x) \ \geq \ g(\x_t) - 2\ell(\x - \x_t)^\top(\x_t - \tilde{\x}_{t-1}) + \frac{\mu\|\x - \x_t\|^2}{4} - (2\ell + \mu)(1+2\kappa)\|\x_t - \x_t^\star\|^2 - \delta. 
\end{equation*}
Furthermore, we have
\begin{equation*}
(\x - \x_t)^\top(\x_t - \tilde{\x}_{t-1}) \ = \ (\x - \tilde{\x}_{t-1})^\top(\x_t - \tilde{\x}_{t-1}) - \|\x_t - \tilde{\x}_{t-1}\|^2, 
\end{equation*}
and 
\begin{equation*}
\|\x_t - \x_t^\star\|^2 \ \leq \ \frac{2}{\mu + 2\ell}\left(g(\x_t) + \ell\|\x_t - \tilde{\x}_{t-1}\|^2 - \min_{\x \in \XCal} \left\{g(\x) + \ell\|\x - \tilde{\x}_{t-1}\|^2\right\}\right) \ \leq \ \frac{2\delta}{\mu + 2\ell}.
\end{equation*}
Putting these pieces together with $\kappa \geq 1$ yields the desired inequality. 
\end{proof}
The remaining proof is based on Lemma~\ref{Lemma:inexact-APPA}. Indeed, we have 
\begin{eqnarray*}
& & \left(1 - \frac{1}{2\sqrt{\kappa}}\right)g(\x_{t-1}) + \frac{1}{2\sqrt{\kappa}}\left(g(\x^\star) + 14\kappa^{3/2}\delta\right) \\
& \overset{\textnormal{Eq.~\eqref{inequality-inexact-APPA}}}{\geq} & \left(1 - \frac{1}{2\sqrt{\kappa}}\right)\left(g(\x_t) - 2\ell(\x_{t-1} - \tilde{\x}_{t-1})^\top(\x_t - \tilde{\x}_{t-1}) + 2\ell\|\x_t - \tilde{\x}_{t-1}\|^2 + \frac{\mu\|\x_{t-1} - \x_t\|^2}{4} - 7\kappa\delta\right) \\
& & + \frac{1}{2\sqrt{\kappa}}\left(g(\x_t) - 2\ell(\x^\star - \tilde{\x}_{t-1})^\top(\x_t - \tilde{\x}_{t-1}) + 2\ell\|\x_t - \tilde{\x}_{t-1}\|^2 + \frac{\mu\|\x^\star - \x_t\|^2}{4} - 7\kappa\delta\right) + 7\kappa\delta \\
& = & g(\x_t) - 2\ell\left(\left(1 - \frac{1}{2\sqrt{\kappa}}\right)\x_{t-1} + \frac{\x^\star}{2\sqrt{\kappa}} - \tilde{\x}_{t-1}\right)^\top(\x_t - \tilde{\x}_{t-1}) + 2\ell\|\x_t - \tilde{\x}_{t-1}\|^2 + \frac{\mu\|\x^\star - \x_t\|^2}{8\sqrt{\kappa}}. 
\end{eqnarray*}
Equivalently, we have
\begin{eqnarray}\label{inexact-APPA-inequality-first}
g(\x_t) - g(\x^\star) & \leq & \left(1 - \frac{1}{2\sqrt{\kappa}}\right)\left(g(\x_{t-1}) - g(\x^\star)\right) + 2\ell\left(\left(1 - \frac{1}{2\sqrt{\kappa}}\right)\x_{t-1} + \frac{\x^\star}{2\sqrt{\kappa}} - \tilde{\x}_{t-1}\right)^\top(\x_t - \tilde{\x}_{t-1}) \nonumber \\ 
& & - 2\ell\|\x_t - \tilde{\x}_{t-1}\|^2 - \frac{\mu\|\x^\star - \x_t\|^2}{8\sqrt{\kappa}} + 7\kappa\delta. 
\end{eqnarray}
Consider $\tilde{\x}_t = \x_t + \frac{2\sqrt{\kappa}-1}{2\sqrt{\kappa}+1}(\x_t - \x_{t-1})$, we let $\w_t = \tilde{\x}_t + 2\sqrt{\kappa}(\tilde{\x}_t - \x_t)$ and obtain that
\begin{eqnarray*}
\w_t & = & (1 + 2\sqrt{\kappa})\tilde{\x}_t - 2\sqrt{\kappa}\x_t \ = \ 2\sqrt{\kappa}\x_t - (2\sqrt{\kappa}-1)\x_{t-1} \ = \ \left(1 - \frac{1}{2\sqrt{\kappa}}\right)\w_{t-1} + 2\sqrt{\kappa}\x_t - \frac{4\kappa-1}{2\sqrt{\kappa}}\tilde{\x}_{t-1} \\
& = & \left(1 - \frac{1}{2\sqrt{\kappa}}\right)\w_{t-1} + 2\sqrt{\kappa}\left(\x_t - \tilde{\x}_{t-1}\right) + \frac{\tilde{\x}_{t-1}}{2\sqrt{\kappa}}. 
\end{eqnarray*}
This implies that 
\begin{eqnarray}\label{inexact-APPA-inequality-second}
\|\w_t - \x^\star\|^2 & = & \left\|\left(1 - \frac{1}{2\sqrt{\kappa}}\right)\w_{t-1} + \frac{\tilde{\x}_{t-1}}{2\sqrt{\kappa}} - \x^\star + 2\sqrt{\kappa}\left(\x_t - \tilde{\x}_{t-1}\right)\right\|^2 \\ 
& & \hspace*{-6em} = \left\|\left(1 - \frac{1}{2\sqrt{\kappa}}\right)\w_{t-1} + \frac{\tilde{\x}_{t-1}}{2\sqrt{\kappa}} - \x^\star\right\|^2 + 4\sqrt{\kappa}\left(\left(1 - \frac{1}{2\sqrt{\kappa}}\right)\w_{t-1} + \frac{\tilde{\x}_{t-1}}{2\sqrt{\kappa}} - \x^\star\right)^\top\left(\x_t - \tilde{\x}_{t-1}\right) \nonumber \\
& & \hspace*{-5em} + 4\kappa\|\x_t - \tilde{\x}_{t-1}\|^2. \nonumber
\end{eqnarray}
Since $\w_{t-1} = \tilde{\x}_{t-1} + 2\sqrt{\kappa}(\tilde{\x}_{t-1} - \x_{t-1})$, we have 
\begin{equation}\label{inexact-APPA-inequality-third}
(1 - \frac{1}{2\sqrt{\kappa}})\w_{t-1} + \frac{\tilde{\x}_{t-1}}{2\sqrt{\kappa}} \ = \ 2\sqrt{\kappa}\tilde{\x}_{t-1} - (2\sqrt{\kappa}-1)\x_{t-1}. 
\end{equation}
Using the Young's inequality, we have
\begin{eqnarray}\label{inexact-APPA-inequality-fourth}
& & \left\|\left(1 - \frac{1}{2\sqrt{\kappa}}\right)\w_{t-1} + \frac{\tilde{\x}_{t-1}}{2\sqrt{\kappa}} - \x^\star\right\|^2 \\
& \leq & \left(1 - \frac{1}{2\sqrt{\kappa}}\right)^2\left(1 + \frac{5}{8\sqrt{\kappa} -5}\right)\|\w_{t-1} - \x^\star\|^2 + \frac{1}{4\kappa}\left(1 + \frac{8\sqrt{\kappa} -5}{5}\right)\|\tilde{\x}_{t-1} - \x^\star\|^2 \nonumber \\ 
& \leq & \left(1 - \frac{1}{2\sqrt{\kappa}}\right)\left(1 + \frac{1}{8\sqrt{\kappa} -5}\right)\|\w_{t-1} - \x^\star\|^2 + \frac{2\|\tilde{\x}_{t-1} - \x^\star\|^2}{5\sqrt{\kappa}}  \nonumber \\ 
& \leq & \left(1 - \frac{1}{6\sqrt{\kappa}}\right)\|\w_{t-1} - \x^\star\|^2 + \frac{2\|\tilde{\x}_{t-1} - \x^\star\|^2}{5\sqrt{\kappa}}. \nonumber
\end{eqnarray}
Using the Young's inequality again, we have
\begin{equation}\label{inexact-APPA-inequality-fifth}
\|\tilde{\x}_{t-1} - \x^\star\|^2 \ \leq \ \frac{5\|\x^\star - \x_t\|^2}{4} + 5\|\tilde{\x}_{t-1} - \x_t\|^2. 
\end{equation}
Putting Eq.~\eqref{inexact-APPA-inequality-second}-Eq.~\eqref{inexact-APPA-inequality-fifth} together with $\kappa \geq 1$, we have
\begin{eqnarray}\label{inexact-APPA-inequality-sixth}
\|\w_t - \x^\star\|^2 & \leq & \left(1 - \frac{1}{6\sqrt{\kappa}}\right)\|\w_{t-1} - \x^\star\|^2 + \frac{\|\x^\star - \x_t\|^2}{2\sqrt{\kappa}} + 6\kappa\|\x_t - \tilde{\x}_{t-1}\|^2 \\
& & + 8\kappa\left(\tilde{\x}_{t-1} - \left(1 - \frac{1}{2\sqrt{\kappa}}\right)\x_{t-1} - \frac{\x^\star}{2\sqrt{\kappa}}\right)^\top\left(\x_t - \tilde{\x}_{t-1}\right). \nonumber
\end{eqnarray}
Combining Eq.~\eqref{inexact-APPA-inequality-first} and Eq.~\eqref{inexact-APPA-inequality-sixth} yields that 
\begin{eqnarray*}
g(\x_t) - g(\x^\star) + \frac{\mu\|\w_t - \x^\star\|^2}{4} & \leq & \left(1 - \frac{1}{2\sqrt{\kappa}}\right)\left(g(\x_{t-1}) - g(\x^\star)\right) + \left(1 - \frac{1}{6\sqrt{\kappa}}\right)\frac{\mu\|\w_{t-1} - \x^\star\|^2}{4} + 7\kappa\delta \\
& \leq & \left(1 - \frac{1}{6\sqrt{\kappa}}\right)\left(g(\x_{t-1}) - g(\x^\star) + \frac{\mu\|\w_{t-1} - \x^\star\|^2}{4}\right) + 7\kappa\delta. 
\end{eqnarray*}
Repeating the above inequality yields that 
\begin{equation*}
g(\x_T) - g(\x^\star) + \frac{\mu\|\w_T - \x^\star\|^2}{4} \ \leq \ \left(1 - \frac{1}{6\sqrt{\kappa}}\right)^T\left(g(\x_0) - g(\x^\star) + \frac{\mu\|\x_0 - \x^\star\|^2}{4}\right) + 42\kappa^{3/2}\delta. 
\end{equation*}
Therefore, we conclude that
\begin{equation*}
g(\x_T) - g(\x^\star) \ \leq \ \left(1 - \frac{1}{6\sqrt{\kappa}}\right)^T\left(g(\x_0) - g(\x^\star) + \frac{\mu\|\x_0 - \x^\star\|^2}{4}\right) + 42\kappa^{3/2}\delta. 
\end{equation*}
Since the tolerance $\delta \leq \epsilon\kappa^{-3/2}/84$, we conclude that the iteration complexity of Algorithm~\ref{alg:APPA} to guarantee that $g(\x_T) - \min_{\x \in \XCal} g(\x) \leq \epsilon$ if there exists an absolute constant $c>0$ such that 
\begin{equation*}
T \geq \ c\sqrt{\kappa}\log\left(\frac{g(\x_0) - g(\x^\star) + (\mu/4)\|\x_0 - \x^\star\|^2}{\epsilon}\right). 
\end{equation*}
This completes the proof. 

\subsection{Proof of Theorem~\ref{Theorem:maximin-AG2}}
Before presenting the main proof, we define the following important functions: 
\begin{equation*}
\begin{array}{ll}
\Phi_g(\cdot) \ = \ \max_{\y \in \YCal} \ g(\cdot, \y), & \qquad \y_g^\star(\cdot) \ = \ \argmax_{\y \in \YCal} \ g(\cdot, \y), \\
\Psi_g(\cdot) \ = \ \min_{\x \in \XCal} \ g(\x, \cdot), & \qquad \x_g^\star(\cdot) \ = \ \argmin_{\x \in \XCal} \ g(\x, \cdot). 
\end{array}
\end{equation*}
All the above functions are well defined since $g(\cdot, \cdot)$ is strongly convex-concave. We provide their complete characterization in the following structural lemma. 
\begin{lemma}\label{Lemma:minimaxAG2-structure}
Under the assumptions imposed in Theorem~\ref{Theorem:maximin-AG2}, we have
\begin{enumerate}[(a)]
\item A function $\y_g^\star(\cdot)$ is $\kappay$-Lipschitz.
\item A function $\Phi_g(\cdot)$ is $2\kappay\ell$-smooth and $\mu_\x$-strongly convex with $\nabla\Phi_g(\cdot) = \gradx g(\cdot, \y_g^\star(\cdot))$. 
\item A function $\x_g^\star(\cdot)$ is $\kappax$-Lipschitz. 
\item A function $\Psi_g(\cdot)$ is $2\kappax\ell$-smooth and $\mu_\y$-strongly concave with $\nabla\Psi_g(\cdot) = \grady g(\x_g^\star(\cdot), \cdot)$. 
\end{enumerate}
where $\kappax = \ell/\mu_\x$ and $\kappay = \ell/\mu_\y$ are condition numbers. 
\end{lemma}
Now we are ready to prove Theorem~\ref{Theorem:maximin-AG2}. We divide the proof into three parts. In the first part, we show that the output $\hat{\x} = \textsc{\AGPROX}(g, \x_0, \y_0, \ell, \mu_\x, \mu_\y, \epsilon)$ satisfies 
\begin{equation}\label{criterion:maximin-AG2}
\max_{\y\in\YCal} \ g(\hat{\x}, \y) \ \leq \ \min_{\x \in \XCal} \max_{\y \in \YCal} \ g(\x, \y) + \epsilon
\end{equation}
In the second part, we get the sufficient condition for guaranteeing the stopping criteria in Algorithm~\ref{alg:Proxy}. In the third part, we estimate an upper bound for the gradient complexity of the algorithm using the condition derived in the second part. For the ease of presentation, we denote $(\x_g^\star, \y_g^\star)$ as the unique solution to the minimax optimization $\min_{\x \in \XCal} \max_{\y \in \YCal} g(\x, \y)$.    

\paragraph{Part I.} By the definition of $\Phi_g$, the inequality in Eq.~\eqref{criterion:maximin-AG2} can be rewritten as follows, 
\begin{equation*}
\Phi_g(\hat{\x}) \ \leq \ \min_{\x \in \XCal} \ \Phi_g(\x) + \epsilon.  
\end{equation*}
Since $\hat{\x} = \proj_\XCal(\x_T - (1/2\kappa_\y\ell)\gradx g(\x_T, \y_T))$, we have 
\begin{eqnarray*}
0 & \leq & (\x - \hat{\x})^\top\left(2\kappa_\y\ell(\hat{\x} - \x_T) + \gradx g(\x_T, \y_T)\right) \\
& = & (\x - \hat{\x})^\top(2\kappa_\y\ell(\hat{\x} - \x_T) + \nabla\Phi_g(\x_T)) + (\x - \hat{\x})^\top(\gradx g(\x_T, \y_T) - \nabla\Phi_g(\x_T)). 
\end{eqnarray*}
Since $\grad\Phi_g(\x_T) = \gradx g(\x_T, \y_g^\star(\x_T))$, we have $\|\gradx g(\x_T, \y_T) - \nabla\Phi_g(\x_T)\| \leq \ell\|\y_T - \y_g^\star(\x_T)\|$. Using the Young's inequality, we have
\begin{equation*}
(\x - \hat{\x})^\top(\gradx g(\x_T, \y_T) - \nabla\Phi_g(\x_T)) \ \leq \ \frac{\kappay\ell\|\hat{\x} - \x_T\|^2}{2} + \frac{\kappay\ell\|\x - \x_T\|^2}{2} + \mu_\y\|\y_T - \y_g^\star(\x_T)\|^2. 
\end{equation*}
Since $\Phi_g$ is $2\kappay\ell$-smooth and $\mu_\x$-strongly convex, we have
\begin{eqnarray*}
(\x - \hat{\x})^\top(2\kappay\ell(\hat{\x} - \x_T) + \grad \Phi_g(\x_T)) & \leq & 2\kappay\ell(\x - \x_T)^\top(\hat{\x} - \x_T) + \Phi_g(\x) - \Phi_g(\hat{\x}) \\
& & \hspace*{-4em} - \kappay\ell\|\hat{\x} - \x_T\|^2 - \frac{\mu_\x\|\x - \x_T\|^2}{2}.
\end{eqnarray*}
Using the Young's inequality, we have $(\x - \x_T)^\top(\hat{\x} - \x_T) \leq \|\x - \x_T\|^2 + (1/4)\|\hat{\x} - \x_T\|^2$. Putting these pieces together yields with $\x = \x_g^\star$ yields that 
\begin{equation}\label{maximin-AG2-inequality-main}
\Phi_g(\hat{\x}) - \min_{\x \in \XCal} \ \Phi_g(\x) \ \leq \ 3\kappa_\y\ell\|\x_T - \x_g^\star\|^2 + \mu_\y\|\y_T - \y_g^\star(\x_T)\|^2. 
\end{equation}
In what follows, we prove that $\Phi_g(\hat{\x}) \leq \min_{\x \in \XCal} \Phi_g(\x) + \epsilon$ if the following stopping conditions hold true,  
\begin{align}
g(\x_T, \y_T) - g(\x_g^\star(\y_T), \y_T) & \leq \ \frac{\epsilon}{648\kappax^3\kappay^3}, \label{criterion:x-update} \\
\|\y_T - \proj_\YCal(\y_T + (1/2\kappax\ell)\grady g(\x_T, \y_T))\| & \leq \ \frac{1}{24\kappax^2\kappay}\sqrt{\frac{\epsilon}{\kappay\ell}}. \label{criterion:y-update}
\end{align}
Indeed, we observe that $\|\x_T - \x_g^\star\| \leq \|\x_T - \x_g^\star(\y_T)\| + \|\x_g^\star(\y_T) - \x_g^\star(\y_g^\star)\| + \|\x_g^\star(\y_g^\star) - \x_g^\star\|$. By definition, we have $\x_g^\star(\y_g^\star) = \x_g^\star$. Also, $\x_g^\star(\cdot)$ is $\kappax$-Lipschitz. Therefore, we have
\begin{equation}\label{maximin-AG2-inequality-first-x}
\|\x_T - \x_g^\star\| \ \leq \ \|\x_T - \x_g^\star(\y_T)\| + \kappax\|\y_T - \y_g^\star\|. 
\end{equation}
By the similar argument, we have
\begin{equation}\label{maximin-AG2-inequality-first-y}
\|\y_T - \y_g^\star(\x_T)\| \ \leq \ \|\y_T - \y_g^\star\| + \kappa_\y\|\x_T - \x_g^\star\| \ \leq \ \kappay\|\x_T - \x_g^\star(\y_T)\| + \kappax\kappay\|\y_T - \y_g^\star\|. 
\end{equation}
First, we bound the term $\|\x_T - \x_g^\star(\y_T)\|$. Since $g(\cdot, \y_T)$ is $\mu_\x$-strongly convex, we have
\begin{equation}\label{maximin-AG2-inequality-second}
\|\x_T - \x_g^\star(\y_T)\| \ \leq \ \sqrt{\frac{2(g(\x_T, \y_T) - g(\x^\star(\y_T), \y_T))}{\mu_\x}} \ \leq \ \frac{1}{18\kappax\kappay}\sqrt{\frac{\epsilon}{\kappay\ell}}
\end{equation}
It remains to bound the term $\|\y_T - \y_g^\star\|$. Indeed, we have $\grad\Psi_g(\y_T) = \grady g(\x_g^\star(\y_T), \y_T)$ and 
\begin{eqnarray*}
\|\y_T - \PY(\y_T + (1/2\kappax\ell)\nabla \Psi_g(\y_T))\| & \leq & \|\y_T - \PY(\y_T + (1/2\kappax\ell)\grady g(\x_T, \y_T))\| \\
& & \hspace{-6em} + \|\PY(\y_T + (1/2\kappax\ell)\grady g(\x_T, \y_T)) - \PY(\y_T + (1/2\kappax\ell)\nabla \Psi_g(\y_T))\|.
\end{eqnarray*}
Since $\PY$ is nonexpansive and $\grady g$ is $\ell$-Lipschitz, we have
\begin{equation*}
\|\PY(\y_T + (1/2\kappax\ell)\grady g(\x_T, \y_T)) - \PY(\y_T + (1/2\kappax\ell)\nabla \Psi_g(\y_T))\| \ \leq \ \frac{\|\x_T - \x_g^\star(\y_T)\|}{2\kappax}. 
\end{equation*}
Putting these pieces together with Eq.~\eqref{criterion:y-update} and Eq.~\eqref{maximin-AG2-inequality-second} yields that 
\begin{equation}\label{maximin-AG2-inequality-third}
\|\y_T - \PY(\y_T + (1/2\kappax\ell)\nabla \Psi_g(\y_T))\| \ \leq \ \frac{1}{18\kappax^2\kappay}\sqrt{\frac{\epsilon}{\kappay\ell}}. 
\end{equation}
Since $\y_g^\star = \argmax_{\y \in \YCal} \Psi_g(\y)$ and $\tilde{\y}_T = \PY(\y_T + (1/2\kappax\ell)\grad \Psi_g(\y_T))$ is achieved by one-step projected gradient ascent from $\y_T$, we derive from the $2\kappax\ell$-smoothness of $\Psi_g$, we have
\begin{equation}\label{maximin-AG2-inequality-fourth}
\|\tilde{\y}_T - \y_g^\star\| \ \leq \ \|\y_T - \y_g^\star\|. 
\end{equation}
Using the definition of $\tilde{\y}_T$ and $\y_g^\star$, we have
\begin{equation*}
(\y_g^\star - \tilde{\y}_T)^\top\left(\tilde{\y}_T - \y_T - (1/2\kappax\ell) \grad\Psi_g(\y_T)\right) \ \geq \ 0, \qquad (\y_g^\star - \tilde{\y}_T)^\top\grad\Psi_g(\y_g^\star) \ \geq \ 0. 
\end{equation*}
Summing up the above two inequalities and rearranging yields that 
\begin{equation*}
(\y_g^\star - \y_T)^\top(\grad\Psi_g(\y_g^\star) - \grad\Psi_g(\y_T)) \ \geq \ 2\kappax\ell(\y_g^\star - \tilde{\y}_T)^\top(\y_T - \tilde{\y}_T) + (\tilde{\y}_T - \y_T)^\top(\grad\Psi_g(\y_g^\star) - \grad\Psi_g(\y_T)). 
\end{equation*}
Since $\Psi_g$ is $2\kappax\ell$-smooth and $\mu_\y$-strongly concave, we have
\begin{equation*}
- \mu_\y\|\y_g^\star - \y_T\|^2 \ \geq \ - 2\kappax\ell\|\tilde{\y}_T - \y_T\|\left(\|\y_g^\star - \tilde{\y}_T\| + \|\y_g^\star - \y_T\|\right) \ \overset{~\eqref{maximin-AG2-inequality-fourth}}{\geq} \ - 4\kappax\ell\|\tilde{\y}_T - \y_T\|\|\y_g^\star - \y_T\|. 
\end{equation*}
This implies that 
\begin{equation}\label{maximin-AG2-inequality-fifth}
\|\y_g^\star - \y_T\| \ \leq \ 4\kappax\kappay\|\y_T - \tilde{\y}_T\| \ \overset{~\eqref{maximin-AG2-inequality-third}}{\leq} \ \frac{1}{4\kappax}\sqrt{\frac{\epsilon}{\kappay\ell}}. 
\end{equation}
Plugging Eq.~\eqref{maximin-AG2-inequality-second} and Eq.~\eqref{maximin-AG2-inequality-fifth} into Eq.~\eqref{maximin-AG2-inequality-first-x} yields that 
\begin{equation*}
\|\x_T - \x_g^\star\| \ \leq \ \left(\frac{1}{18\kappax\kappay} + \frac{1}{4}\right)\sqrt{\frac{\epsilon}{\kappay\ell}} \ \overset{\kappax, \kappay \geq 1}{\leq} \ \frac{1}{2}\sqrt{\frac{\epsilon}{\kappay\ell}}. 
\end{equation*}
Plugging Eq.~\eqref{maximin-AG2-inequality-second} and Eq.~\eqref{maximin-AG2-inequality-fifth} into Eq.~\eqref{maximin-AG2-inequality-first-y} yields that 
\begin{equation*}
\|\y_T - \y_g^\star(\x_T)\| \ \leq \ \left(\frac{1}{18\kappax} + \frac{\kappa_\y}{4}\right)\sqrt{\frac{\epsilon}{\kappay\ell}} \ \overset{\kappax, \kappay \geq 1}{\leq} \ \frac{1}{2}\sqrt{\frac{\kappay\epsilon}{\ell}}. 
\end{equation*}
Putting these pieces together Eq.~\eqref{maximin-AG2-inequality-main} yields the desired result. 
\paragraph{Part II.} We first show that $\|\y_T - \y_g^\star\| \leq (1/216\kappax^2\kappay)\sqrt{\epsilon/\kappay\ell}$ and Eq.~\eqref{criterion:x-update} are sufficient to guarantee Eq.~\eqref{criterion:y-update}. Indeed, we have $\y_g^\star = \proj_\YCal(\y_g^\star + (1/2\kappax\ell)\grad\Psi_g(\y_g^\star))$. This together with the triangle inequality and the nonexpansiveness of $\PY$ yields
\begin{equation*}
\|\y_T - \PY(\y_T + (1/2\kappax\ell)\grady g(\x_T, \y_T))\| \ \leq \ 2\|\y_T - \y_g^\star\| + \frac{\|\grady g(\x_T, \y_T) - \grad\Psi_g(\y_g^\star)\|}{2\kappax\ell}. 
\end{equation*}
Furthermore, $\nabla\Psi_g(\y_T) = \grady g(\x^\star(\y_T), \y_T)$ and 
\begin{equation*}
\|\grady g(\x_T, \y_T) - \grad\Psi_g(\y_g^\star)\| \ \leq \ \|\grady g(\x_T, \y_T) - \grady g(\x_g^\star(\y_T), \y_T)\| + \|\grad\Psi_g(\y_T) - \grad\Psi_g(\y_g^\star)\|. 
\end{equation*}
Since $g$ is $\ell$-smooth and $\Psi_g$ is $2\kappax\ell$-smooth, we have
\begin{equation*}
\|\grady g(\x_T, \y_T) - \grad\Psi_g(\y_g^\star)\| \ \leq \ \ell\|\x_T - \x_g^\star(\y_T)\| + 2\kappax\ell\|\y_T - \y_g^\star\|. 
\end{equation*}
Also, Eq.~\eqref{criterion:x-update} guarantees that Eq.~\eqref{maximin-AG2-inequality-second} holds true. Then we have
\begin{equation*}
\|\y_T - \proj_\YCal(\y_T + (1/2\kappax\ell)\grady g(\x_T, \y_T))\| \ \leq \ 3\|\y_T - \y_g^\star\| + \frac{1}{36\kappax^2\kappay}\sqrt{\frac{\epsilon}{\kappay\ell}}. 
\end{equation*}
The above inequality together with $\|\y_T - \y_g^\star\| \leq (1/216\kappax^2\kappay)\sqrt{\epsilon/\kappay\ell}$ guarantees Eq.~\eqref{criterion:y-update}. Next we derive a sufficient condition for guaranteeing $\|\y_T - \y_g^\star\| \leq (1/216\kappax^2\kappay)\sqrt{\epsilon/\kappay\ell}$. Since $\Psi_g$ is $\mu_\y$-strongly concave,~\citet[Theorem~2.1.5]{Nesterov-2018-Lectures} implies that 
\begin{equation*}
\|\y_T - \y_g^\star\|^2 \ \leq \ \frac{2}{\mu_\y} \left(\max_{\y \in \YCal} \Psi_g(\y) - \Psi_g(\y_T)\right). 
\end{equation*}
Putting these pieces together yields the desired condition as follows, 
\begin{equation}\label{maximin-AG2-inequality-sixth}
\max_{\y \in \YCal} \ \Psi_g(\y) - \Psi_g(\y_T) \ \leq \ \frac{\epsilon}{93312\kappax^4\kappay^4}. 
\end{equation}
\paragraph{Part III.} We proceed to estimate an upper bound for the gradient complexity of Algorithm~\ref{alg:Proxy} using Eq.~\eqref{maximin-AG2-inequality-sixth}. Note that $\tilde{\epsilon} \leq \epsilon/(4477676(\kappax\kappay)^{11/2})$ and we provide a key technical lemma which is crucial to the subsequent analysis.   
\begin{lemma}\label{lemma:maximin-AG2}
For any $\y \in \YCal$ and $\{(\y_t, \tilde{\y}_t)\}_{t \geq 0}$ generated by Algorithm~\ref{alg:Proxy}, we have
\begin{equation*}
\Psi_g(\y) \ \leq \ 2\kappax\ell(\y - \tilde{\y}_{t-1})^\top(\y_t - \tilde{\y}_{t-1}) + \Psi_g(\y_t) - \frac{\kappax\ell\|\y_t - \tilde{\y}_{t-1}\|^2}{2} - \frac{\mu_\y\|\y - \tilde{\y}_{t-1}\|^2}{4} + 3\kappax\kappay\tilde{\epsilon}. 
\end{equation*}
\end{lemma}
\begin{proof}
For any $\y \in \YCal$, the update formula $\y_t \leftarrow \PY(\tilde{\y}_{t-1} + (1/2\kappax\ell)\grady g(\tilde{\x}_{t-1}, \tilde{\y}_{t-1}))$ implies that 
\begin{eqnarray*}
0 & \leq & (\y - \y_t)^\top(2\kappax\ell(\y_t - \tilde{\y}_{t-1}) - \grady g(\tilde{\x}_{t-1}, \tilde{\y}_{t-1})) \\
& = & (\y - \y_t)^\top(2\kappax\ell(\y_t - \tilde{\y}_{t-1}) - \grad \Psi_g(\tilde{\y}_{t-1})) + (\y - \y_t)^\top(\grad\Psi_g(\tilde{\y}_{t-1}) - \grady g(\tilde{\x}_{t-1}, \tilde{\y}_{t-1})). 
\end{eqnarray*}
Since $\grad\Psi_g(\tilde{\y}_{t-1}) = \grady g(\x_g^\star(\tilde{\y}_{t-1}), \tilde{\y}_{t-1})$, we have
\begin{equation*}
\|\grad\Psi_g(\tilde{\y}_{t-1}) - \grady g(\tilde{\x}_{t-1}, \tilde{\y}_{t-1})\| \ \leq \ \ell\|\x_g^\star(\tilde{\y}_{t-1}) - \tilde{\x}_{t-1}\|. 
\end{equation*}
Since $g(\cdot, \tilde{\y}_{t-1})$ is $\mu_\x$-strongly convex, we have
\begin{equation*}
\|\x_g^\star(\tilde{\y}_{t-1}) - \tilde{\x}_{t-1}\| \ \leq \ \sqrt{\frac{2(g(\tilde{\x}_{t-1}, \tilde{\y}_{t-1}) - g(\x_g^\star(\tilde{\y}_{t-1}), \tilde{\y}_{t-1}))}{\mu_\x}} \ \leq \  \sqrt{\frac{2\tilde{\epsilon}}{\mu_\x}}.  
\end{equation*}
Using Young's inequality, we have
\begin{equation*}
(\y - \y_t)^\top(\grad\Psi_g(\tilde{\y}_{t-1}) - \grady g(\tilde{\x}_{t-1}, \tilde{\y}_{t-1})) \ \leq \ \frac{\kappax\ell\|\y_t - \tilde{\y}_{t-1}\|^2}{2} + \frac{\mu_\y\|\y - \tilde{\y}_{t-1}\|^2}{4} + 3\kappax\kappay\tilde{\epsilon}. 
\end{equation*}
Since $\Psi_g$ is $2\kappax\ell$-smooth and $\mu_\y$-strongly concave, we have
\begin{eqnarray*}
(\y - \y_t)^\top(2\kappax\ell(\y_t - \tilde{\y}_{t-1}) - \grad \Psi_g(\tilde{\y}_{t-1})) & \leq & 2\kappax\ell(\y - \tilde{\y}_{t-1})^\top(\y_t - \tilde{\y}_{t-1}) + \Psi_g(\y_t) - \Psi_g(\y) \\
& & \hspace*{-4em} - \kappax\ell\|\y_t - \tilde{\y}_{t-1}\|^2 - \frac{\mu_\y\|\y - \tilde{\y}_{t-1}\|^2}{2}. 
\end{eqnarray*}
Putting these pieces together yields the desired inequality. 
\end{proof}
The remaining proof is based on the modification of Nesterov's techniques~\citep[Section~2.2.5]{Nesterov-2018-Lectures}. Indeed, we define the estimate sequence as follows, 
\begin{eqnarray*}
\Gamma_0(\y) & = & \Psi_g(\y_0) - \frac{\mu_\y\|\y - \y_0\|^2}{2}, \\
\Gamma_{t+1}(\y) & = & \frac{1}{4\sqrt{\kappax\kappay}}\left(\Psi_g(\y_{t+1}) + 2\kappax\ell(\y - \tilde{\y}_t)^\top(\y_{t+1} - \tilde{\y}_t) - \frac{\kappax\ell\|\y_{t+1} - \tilde{\y}_t\|^2}{2} \right. \\
& & \left. - \frac{\mu_\y\|\y - \tilde{\y}_t\|^2}{4} - 12(\kappax\kappay)^{3/2}\tilde{\epsilon}\right) + \left(1 - \frac{1}{4\sqrt{\kappax\kappay}}\right)\Gamma_t(\y) \quad \text{for all } t \geq 0.  
\end{eqnarray*}
We apply the inductive argument to prove,  
\begin{equation}\label{maximin-AG2-inequality-seventh}
\max_{\y \in \br^n} \ \Gamma_t(\y) \ \leq \ \Psi_g(\y_t) \ \quad \text{for all } t \geq 0.  
\end{equation}
Eq.~\eqref{maximin-AG2-inequality-seventh} holds trivially when $t = 0$. In what follows, we show that Eq.~\eqref{maximin-AG2-inequality-seventh} holds true when $t=T$ if Eq.~\eqref{maximin-AG2-inequality-seventh} holds true for all $t \leq T-1$. Let $\sv_t = \argmax_{\y \in \br^n} \Gamma_t(\y)$ and $\Gamma_t^\star = \max_{\y \in \br^n} \Gamma_t(\y)$, we have the canonical form $\Gamma_t(\y) = \Gamma_t^\star - (\mu_\y/4)\|\y - \sv_t\|^2$. The following recursive rules hold for $\sv_t$ and $\Gamma_t^\star$: 
\begin{eqnarray*}
\sv_{t+1} & = & \left(1 - \frac{1}{4\sqrt{\kappax\kappay}}\right)\sv_t + \frac{\tilde{\y}_t}{4\sqrt{\kappax\kappay}} + \sqrt{\kappax\kappay}(\y_{t+1} - \tilde{\y}_t), \\
\Gamma_{t+1}^\star & = & \left(1 - \frac{1}{4\sqrt{\kappax\kappay}}\right)\Gamma_t^\star + \frac{1}{4\sqrt{\kappax\kappay}}\left(\Psi_g(\y_{t+1}) - 12(\kappax\kappay)^{3/2}\tilde{\epsilon}\right) - \left(\frac{\ell}{8}\sqrt{\frac{\kappax}{\kappay}} - \frac{\kappax\ell}{4}\right)\|\y_{t+1} - \tilde{\y}_t\|^2 \\
& & - \frac{1}{4\sqrt{\kappax\kappay}}\left(1 - \frac{1}{4\sqrt{\kappax\kappay}}\right)\left(\frac{\mu_\y\|\tilde{\y}_t - \sv_t\|^2}{4} - 2\kappax\ell(\sv_t - \tilde{\y}_t)^\top(\y_{t+1} - \tilde{\y}_t)\right). 
\end{eqnarray*}
It follows from the recursive rule for $\Gamma_t$ and its canonical form that 
\begin{equation*}
\grad\Gamma_{t+1}(\y) \ = \ -\left(1 - \frac{1}{4\sqrt{\kappax\kappay}}\right)\frac{\mu_\y(\y - \sv_t)}{2} + \frac{1}{4\sqrt{\kappax\kappay}}\left(2\kappax\ell(\y_{t+1} - \tilde{\y}_t) - \frac{\mu_\y(\y - \tilde{\y}_t)}{2}\right). 
\end{equation*}
The recursive rule for $\sv_t$ can be achieved by solving $\nabla\Gamma_{t+1}(\sv_{t+1}) = 0$. Then we have
\begin{eqnarray*}
& & \Gamma_{t+1}^\star \ = \ \Gamma_{t+1}(\sv_{t+1}) \\ 
& = & \left(1 - \frac{1}{4\sqrt{\kappax\kappay}}\right)\Gamma_t^\star - \left(1 - \frac{1}{4\sqrt{\kappax\kappay}}\right)\frac{\mu_\y\|\sv_{t+1} - \sv_t\|^2}{4} + \frac{1}{4\sqrt{\kappax\kappay}}\left(\Psi_g(\y_{t+1}) - 12(\kappax\kappay)^{3/2}\tilde{\epsilon} \right. \\
& & \left. - \frac{\kappax\ell\|\y_{t+1} - \tilde{\y}_t\|^2}{2}\right) + \frac{1}{4\sqrt{\kappax\kappay}}\left(2\kappax\ell(\sv_{t+1} - \tilde{\y}_t)^\top(\y_{t+1} - \tilde{\y}_t) - \frac{\mu_\y\|\sv_{t+1} - \tilde{\y}_t\|^2}{4}\right). 
\end{eqnarray*}
Then we conclude the recursive rule for $\Gamma_t^\star$ by plugging the recursive rule for $\sv_k$ into the above equality. By the induction, Eq.~\eqref{maximin-AG2-inequality-seventh} holds true when $t = T-1$ which implies 
\begin{eqnarray*}
\Gamma_T^\star & \leq & \left(1 - \frac{1}{4\sqrt{\kappax\kappay}}\right)\Psi_g(\y_{T-1}) + \frac{1}{4\sqrt{\kappax\kappay}}\left(\Psi_g(\y_T) - 12(\kappax\kappay)^{3/2}\tilde{\epsilon}\right) \\
& & - \left(\frac{\ell}{8}\sqrt{\frac{\kappax}{\kappay}} - \frac{\kappax\ell}{4}\right)\|\y_T - \tilde{\y}_{T-1}\|^2 -\frac{1}{4\sqrt{\kappax\kappay}}\left(1 - \frac{1}{4\sqrt{\kappax\kappay}}\right)\left(\frac{\mu_\y\|\tilde{\y}_{T-1} - \sv_{T-1}\|^2}{2} \right. \\
& & \left. - 2\kappax\ell(\sv_{T-1} - \tilde{\y}_{T-1})^\top(\y_T - \tilde{\y}_{T-1})\right).
\end{eqnarray*}
Applying Lemma~\ref{lemma:maximin-AG2} with $t = T$ and $\y = \y_{T-1}$ further implies that 
\begin{eqnarray*}
\Psi_g(\y_{T-1}) & \leq & 2\kappax\ell(\y_{T-1} - \tilde{\y}_{T-1})^\top(\y_T - \tilde{\y}_{T-1}) + \Psi(\y_T) - \frac{\kappax\ell\|\y_T - \tilde{\y}_{T-1}\|^2}{2} \\ 
& & - \frac{\mu_\y\|\y_{T-1} - \tilde{\y}_{T-1}\|^2}{2} + 3\kappax\kappay\tilde{\epsilon}. 
\end{eqnarray*}
Putting these pieces together yields that 
\begin{eqnarray*}
\Gamma_T^\star & \leq & \Psi_g(\y_T) + \left(1 - \frac{1}{4\sqrt{\kappax\kappay}}\right) 2\kappax\ell(\y_T - \tilde{\y}_{T-1})^\top\left[(\y_{T-1} - \tilde{\y}_{T-1}) + \frac{1}{4\sqrt{\kappax\kappay}}(\sv_{T-1} - \tilde{\y}_{T-1}) \right]. 
\end{eqnarray*}
Using the update formula $\tilde{\y}_t = \y_t + \frac{4\sqrt{\kappax\kappay}-1}{4\sqrt{\kappax\kappay}+1}(\y_t - \y_{t-1})$ and the recursive rule for $\sv_t$ with the inductive argument, it is straightforward that $(\y_t - \tilde{\y}_t) + \frac{1}{4\sqrt{\kappax\kappay}}(\sv_t -\tilde{\y}_t) = 0$ for all $t \geq 0$. This implies that $\Gamma_T^\star \leq \Psi_g(\y_T)$. Therefore, we conclude that Eq.~\eqref{maximin-AG2-inequality-seventh} holds true for all $t \geq 0$. 

On the other hand, Lemma~\ref{lemma:maximin-AG2} and the update formula for $\Gamma_t$ implies that 
\begin{equation*}
\Gamma_{t+1}(\y) \ \geq \ \frac{1}{4\sqrt{\kappax\kappay}}\left(\Psi_g(\y) - 12(\kappax\kappay)^{3/2}\tilde{\epsilon} - 3\kappax\kappay\tilde{\epsilon}\right) + \left(1 - \frac{1}{4\sqrt{\kappax\kappay}}\right)\Gamma_t(\y). 
\end{equation*}
Since $\kappax, \kappay \geq 1$, we have
\begin{equation*}
\Psi_g(\y) - \Gamma_{t+1}(\y) \ \leq \ \left(1 - \frac{1}{4\sqrt{\kappax\kappay}}\right)(\Psi_g(\y) - \Gamma_t(\y)) + 6\kappax\kappay\tilde{\epsilon}. 
\end{equation*}
Repeating the above inequality yields that 
\begin{equation*}
\Psi_g(\y) - \Gamma_T(\y) \ \leq \ \left(1 - \frac{1}{4\sqrt{\kappax\kappay}}\right)^T(\Psi_g(\y) - \Gamma_0(\y)) + 24(\kappax\kappay)^{3/2}\tilde{\epsilon}. 
\end{equation*}
Therefore, we conclude that 
\begin{equation*}
\max_{\y \in \YCal} \Psi_g(\y) - \Psi_g(\y_T) \ \leq \ \left(1 - \frac{1}{4\sqrt{\kappax\kappay}}\right)^T\frac{(2\kappax\ell + \bar{\mu})\Dy^2}{2} + 24(\kappax\kappay)^{3/2}\tilde{\epsilon}. 
\end{equation*}
Since the tolerance $\tilde{\epsilon} \leq \frac{\epsilon}{4477676(\kappax\kappay)^{11/2}}$, we conclude that the iteration complexity Algorithm~\ref{alg:Proxy} to guarantee Eq.~\eqref{maximin-AG2-inequality-sixth} is bounded by $O(\sqrt{\kappax\kappay}\log(\ell\Dy^2/\epsilon))$. 

Now it suffices to establish the gradient complexity of the two AGD subroutines at each iteration. In particular, we use the gradient complexity of the AGD subroutine to guarantee that $g(\hat{\x}) \leq \min_\XCal g(\x) + \epsilon$ is bounded by 
\begin{equation*}
O\left(1 + \sqrt{\kappa}\log\left(\frac{\kappa^3\ell \|\x_0 - \x^\star\|^2}{\epsilon}\right)\right),
\end{equation*}
where $\kappa$ is the condition number of $g$ and $\x^\star$ is the global optimum of $g$ over $\XCal$. Since $\YCal$ is a convex and bounded set, $\{\y_t\}_{t \geq 0}$ is a bounded sequence. Hence $\{\tilde{\y}_t\}_{t \geq 0}$ is also a bounded sequence. Since $\x_g^\star(\cdot)$ is $\kappax$-Lipschitz (cf.\ Lemma~\ref{Lemma:minimaxAG2-structure}), the sequences $\{\x_g^\star(\tilde{\y}_t)\}_{t \geq 0}$ and $\{\x_g^\star(\y_t)\}_{t \geq 0}$ are bounded. Thus, we have
\begin{equation*}
\|\x_0 - \x_g^\star(\y_t)\|^2 = \|\x_0 - \x_g^\star(\tilde{\y}_t)\|^2 = O(\|\x_0 - \x_g^\star(\y_0)\|^2 + \kappax^2\Dy^2). 
\end{equation*}
Putting these pieces together yields that the gradient complexity of every AGD subroutines at each iteration is bounded by $O(\sqrt{\kappax}\log((\kappax^3\ell(\|\x_0 - \x_g^\star(\y_0)\|^2 + \kappax^2\Dy^2)/\tilde{\epsilon}))$. Therefore, the gradient complexity of Algorithm~\ref{alg:Proxy} to guarantee Eq.~\eqref{maximin-AG2-inequality-sixth} is bounded by 
\begin{equation*}
O\left(\kappax\sqrt{\kappay}\cdot \log^2\left(\frac{(\kappax+\kappay)\ell(\tilde{D}_\x^2 + \Dy^2)}{\epsilon}\right)\right),
\end{equation*}
where $\kappa_\x = \ell/\mu_\x$ and $\kappa_\y = \ell/\mu_\y$ are condition numbers, $\tilde{D}_\x = \|\x_0 - \x_g^\star(\y_0)\|$ is the initial distance where $\x_g^\star(\y_0) = \argmin_{\x \in \XCal} g(\x, \y_0)$ and $\Dy > 0$ is the diameter of the constraint set $\YCal$.


\section{Proofs for Convex-Concave Settings}
In this section, we present proofs for all results in Section \ref{sec:result_convex_concave}.

\subsection{Proof of Theorem~\ref{thm:scsc}}
We first show that there exists $T>0$ such that $(\hat{\x}, \hat{\y}) = \textsc{\AGCC}(f, \x_0, \y_0, \ell, \mu_\x, \mu_\y, \epsilon, T)$ is an $\epsilon$-saddle point. Then we estimate the total number of gradient evaluations required to output an $\epsilon$-approximate saddle point. 

First, we note that $\textsc{\AGCC}$ in Algorithm~\ref{alg:Minimax_APG} can be interpreted as an inexact accelerated proximal point algorithm $\textsc{Inexact-APPA}$ with the inner loop solver $\textsc{\AGPROX}$ and $\textsc{AGD}$. Using Theorem~\ref{Theorem:AGD} and Theorem~\ref{Theorem:inexact-APPA}, the point $(\hat{\x}, \hat{\y})$ satisfies 
\begin{equation*}
\max_{\y \in \YCal} f(\hat{\x}, \y) - \min_{\x \in \XCal} \max_{\y \in \YCal} f(\x, \y) \ \leq \ \left(1 - \frac{1}{6\sqrt{\kappax}}\right)^T\left(\Phi(\x_0) - \Phi(\x^\star) + \frac{\mu_\x\|\x^\star - \x_0\|^2}{4}\right) + 42\kappax^{3/2}\delta. 
\end{equation*}
and $\hat{\y} \leftarrow \proj_\YCal\left(\tilde{\y} + (1/2\kappa_\x\ell)\grady f(\hat{\x}, \tilde{\y})\right)$ where $\tilde{\y} \in \YCal$ satisfies that 
\begin{equation*}
\max_{\y \in \YCal} f(\hat{\x}, \y) - f(\hat{\x}, \tilde{\y}) \ \leq \ \tilde{\epsilon}. 
\end{equation*}
We let $\Phi(\cdot) = \max_{\y \in \YCal} f(\cdot, \y)$ and note that $\Phi$ is $\mu_\x$-strongly convex function. Since $f$ is $\mu_\x$-strongly-convex-$\mu_\y$-strongly-concave, the Nash equilibrium $(\x^\star, \y^\star)$ is unique and $\x^\star = \argmin_{\x \in \XCal} \Phi(\x)$. Therefore, we have
\begin{equation*}
\|\hat{\x} - \x^\star\|^2 \ \leq \ \frac{2}{\mu_\x}\left(\max_{\y \in \YCal} f(\hat{\x}, \y) - \min_{\x \in \XCal} \max_{\y \in \YCal} f(\x, \y)\right). 
\end{equation*}
Since $f(\hat{\x}, \cdot)$ is $\mu_\y$-strongly concave,~\citet[Theorem~2.1.5]{Nesterov-2018-Lectures} implies that 
\begin{equation*}
\|\tilde{\y} - \y^\star(\hat{\x})\|^2 \ \leq \ \frac{2}{\mu_\y}\left(\max_{\y \in \YCal} \ f(\hat{\x}, \y) - f(\hat{\x}, \tilde{\y})\right) \ \leq \ \frac{2\tilde{\epsilon}}{\mu_\y}. 
\end{equation*}
Since $\y^\star(\cdot) = \argmax_{\y \in \YCal} f(\cdot, \y)$ is $\kappa_\y$-Lipschitz (cf. Lemma~\ref{Lemma:minimaxAG2-structure}), $\|\y^\star - \y^\star(\hat{\x})\|^2 = \|\y^\star(\x^\star) - \y^\star(\hat{\x})\|^2 \leq \kappa_\y^2\|\hat{\x} - \x^\star\|^2$. Thus, we have
\begin{equation*}
\|\tilde{\y} - \y^\star\|^2 \ \leq \ 2\kappa_\y^2\|\hat{\x} - \x^\star\|^2 + \frac{4\tilde{\epsilon}}{\mu_\y}. 
\end{equation*}
Let $\Psi(\cdot) = \min_{\x \in \XCal} f(\x, \cdot)$. By the definition of $\hat{\y}$, the following inequality holds true for any $\y \in \YCal$, 
\begin{eqnarray*}
0 & \leq & (\y - \hat{\y})^\top(2\kappax\ell(\hat{\y} - \tilde{\y}) - \grady f(\hat{\x}, \tilde{\y})) \\
& = & (\y - \hat{\y})^\top(2\kappax\ell(\hat{\y} - \tilde{\y}) - \grad \Psi(\tilde{\y})) + (\y - \hat{\y})^\top(\grad\Psi(\tilde{\y}) - \grady f(\hat{\x}, \tilde{\y})). 
\end{eqnarray*}
Since $\grad\Psi(\tilde{\y}) = \grady f(\x^\star(\tilde{\y}), \tilde{\y})$, we have $\|\grad\Psi(\tilde{\y}) - \grady f(\hat{\x}, \tilde{\y})\| \leq \ell\|\x^\star(\tilde{\y})-\hat{\x}\|$. Using the Young's inequality, we have
\begin{equation*}
(\y - \hat{\y})^\top(\grad\Psi(\tilde{\y}) - \grady f(\hat{\x}, \tilde{\y})) \ \leq \ \frac{\kappax\ell\|\hat{\y} - \tilde{\y}\|^2}{2} + \frac{\kappax\ell\|\y - \tilde{\y}\|^2}{2} + \mu_\x\|\x^\star(\tilde{\y}) - \hat{\x}\|^2. 
\end{equation*}
Since $\Psi$ is $\mu_\y$-strongly concave and $2\kappax\ell$-smooth, we have
\begin{eqnarray*}
(\y - \hat{\y})^\top(2\kappax\ell(\hat{\y} - \tilde{\y}) - \grad \Psi(\tilde{\y})) & \leq & 2\kappax\ell(\y - \tilde{\y})^\top(\hat{\y} - \tilde{\y}) + \Psi(\hat{\y}) - \Psi(\y) \\
& & \hspace*{-4em} - \kappax\ell\|\hat{\y} - \tilde{\y}\|^2 - \frac{\mu_\x\|\y - \tilde{\y}\|^2}{2}. 
\end{eqnarray*}
Using the Young's inequality, we have $(\y - \tilde{\y})^\top(\hat{\y} - \tilde{\y}) \leq \|\y - \tilde{\y}\|^2 + (1/4)\|\hat{\y} - \tilde{\y}\|^2$. Putting these pieces together with $\y = \y^\star$ yields that 
\begin{eqnarray*}
\min_{\x \in \XCal} \max_{\y \in \YCal} f(\x, \y) - \min_{\x \in \XCal} f(\x, \hat{\y}) & = & \Psi(\y^\star) - \Psi(\hat{\y}) \ \leq \ 3\kappax\ell\|\tilde{\y} - \y^\star\|^2 + \mu_\x\|\x^\star(\tilde{\y}) - \hat{\x}\|^2 \\
& & \hspace*{-10em} \leq \ 3\kappax\ell\|\tilde{\y} - \y^\star\|^2 + 2\mu_\x\|\x^\star(\tilde{\y}) - \x^\star(\y^\star)\|^2 + 2\mu_\x\|\x^\star - \hat{\x}\|^2 \\
& & \hspace*{-10em} \leq \ 5\kappax\ell\|\tilde{\y} - \y^\star\|^2 + 2\mu_\x\|\x^\star - \hat{\x}\|^2. 
\end{eqnarray*}
Therefore, we conclude that 
\begin{equation*}
\max_{\y \in \YCal} f(\hat{\x}, \y) - \min_{\x \in \XCal} f(\x, \hat{\y}) \ \leq \ 20\kappax\kappay\tilde{\epsilon} + (20\kappax^2\kappay^2 + 5)\left(\max_{\y \in \YCal} f(\hat{\x}, \y) - \min_{\x \in \XCal} \max_{\y \in \YCal} f(\x, \y)\right). 
\end{equation*}
Note that $\tilde{\epsilon} \leq \epsilon/80\kappax\kappay$ and $\delta \leq \epsilon/4200\kappax^{7/2}\kappay^2$. This together with the above inequality implies that
\begin{equation*}
\max_{\y \in \YCal} f(\hat{\x}, \y) - \min_{\x \in \XCal} f(\x, \hat{\y}) \ \leq \ \frac{3\epsilon}{4} + (20\kappax^2\kappay^2 + 5)\left(1 - \frac{1}{6\sqrt{\kappax}}\right)^T\left(\Phi(\x_0) - \Phi(\x^\star) + \frac{\mu_\x\|\x^\star - \x_0\|^2}{4}\right). 
\end{equation*}
To this end, there exists an absolute constant $c>0$ such that $\max_{\y \in \YCal} f(\hat{\x}, \y) - \min_{\x \in \XCal} f(\x, \hat{\y}) \leq \epsilon$ if the maximum number of iterations $T \geq c\sqrt{\kappax}\log(\kappax^2\kappay^2\ell\|\x^\star - \x_0\|^2/\epsilon)$. This implies that the total number of iterations is bounded by 
\begin{equation*}
O\left(\sqrt{\kappax}\log\left(\frac{\kappax^2\kappay^2\ell\|\x^\star - \x_0\|^2}{\epsilon}\right)\right). 
\end{equation*}
Furthermore, we call the solver $\textsc{\AGPROX}$ at each iteration. Using Theorem~\ref{Theorem:maximin-AG2} and $\delta = \epsilon/(10\kappax\kappay)^4$, the number of gradient evaluations at each iteration is bounded by  
\begin{equation*}
O\left(\sqrt{\kappay}\log\left(\frac{\kappax^{7/2}\kappay^3\ell(\tilde{D}_\x^2 + \Dy^2)}{\epsilon}\right)\log\left(\frac{\kappax^4\kappay^4\ell\Dy^2}{\epsilon}\right)\right). 
\end{equation*}
Recalling $D = \max\{\Dx, \Dy\} < +\infty$, we conclude that the total number of gradient evaluations is bounded by 
\begin{equation*}
O\left(\sqrt{\kappax\kappay}\log^3\left(\frac{\kappax\kappay\ell D^2}{\epsilon}\right)\right). 
\end{equation*}
This completes the proof. 

\subsection{Proof of Corollary~\ref{cor:scc}}
We first show that $(\hat{\x}, \hat{\y}) = \textsc{\AGCC}(f_{\epsilon, \y}, \x_0, \y_0, \ell, \mu_\x, \epsilon/(4D^2_{\y}), \epsilon/2, T)$ is an $\epsilon$-saddle point. Then we estimate the number of gradient evaluations to output an $\epsilon$-saddle point using Theorem~\ref{thm:scsc}. By the definition of $f_\epsilon$, the output $(\hat{\x}, \hat{\y})$ satisfies
\begin{equation*}
\max_{\y \in \YCal} \ \left\{f(\hat{\x}, \y) - \frac{\epsilon\|\y - \y_0\|^2}{4\Dy^2}\right\} - \min_{\x \in \XCal} \ \left\{f(\x, \hat{\y}) - \frac{\epsilon\|\hat{\y} - \y_0\|^2}{4\Dy^2}\right\} \ \leq \ \frac{\epsilon}{2}. 
\end{equation*}
Since the function $f(\x, \cdot)$ is concave for each $\x \in \XCal$, we have
\begin{equation*}
\max_{\y \in \YCal} \ \left\{f(\hat{\x}, \y) - \frac{\epsilon\|\y - \y_0\|^2}{4\Dy^2}\right\} \ \geq \ \max_{\y \in \YCal} \ f(\x_{T+1}, \y) - \frac{\epsilon}{4}. 
\end{equation*}
On the other hand, we have
\begin{equation*}
\min_{\x \in \XCal} \ \left\{f(\x, \hat{\y}) - \frac{\epsilon\|\hat{\y} - \y_0\|^2}{4\Dy^2}\right\} \ \leq \ \min_{\x \in \XCal} \ f(\x, \hat{\y}) + \frac{\epsilon}{4}. 
\end{equation*}
Putting these pieces together yields that $\max_{\y \in \YCal} f(\hat{\x}, \y) - \min_{\x \in \XCal} f(\x, \hat{\y}) \leq \epsilon$. 

Furthermore, letting $\kappay = 2\ell\Dy^2/\epsilon$ in the gradient complexity bound presented in Theorem~\ref{thm:scsc}, we conclude that the total number of gradient evaluations is bounded by 
\begin{equation*}
O\left(\sqrt{\frac{\kappax\ell}{\epsilon}}\Dy\log^3\left(\frac{\kappax\ell D^2}{\epsilon}\right)\right). 
\end{equation*}
This completes the proof. 

\subsection{Proof of Corollary~\ref{cor:cc}}
We first show that $(\hat{\x}, \hat{\y}) = \textsc{\AGCC}(f_{\epsilon}, \x_0, \y_0, \ell, \epsilon/(4\Dx^2), \epsilon/(4\Dy^2), \epsilon/2, T)$ is an $\epsilon$-saddle point. Then we estimate the number of gradient evaluations to output an $\epsilon$-saddle point using Theorem~\ref{thm:scsc}. By the definition of $f_\epsilon$, the output $(\hat{\x}, \hat{\y})$ satisfies
\begin{equation*}
\max_{\y \in \YCal} \ \left\{f(\hat{\x}, \y) + \frac{\epsilon\|\hat{\x} - \x_0\|^2}{8\Dx^2} - \frac{\epsilon\|\y - \y_0\|^2}{8\Dy^2}\right\} - \min_{\x \in \XCal} \ \left\{f(\x, \hat{\y}) + \frac{\epsilon\|\x - \x_0\|^2}{8\Dx^2} - \frac{\epsilon\|\hat{\y} - \y_0\|^2}{8\Dy^2}\right\} \ \leq \ \frac{\epsilon}{2}. 
\end{equation*}
Since the function $f(\x, \cdot)$ is concave for each $\x \in \XCal$, we have
\begin{equation*}
\max_{\y \in \YCal} \ \left\{f(\hat{\x}, \y) + \frac{\epsilon\|\hat{\x} - \x_0\|^2}{8\Dx^2} - \frac{\epsilon\|\y - \y_0\|^2}{8\Dy^2}\right\} \ \geq \ \max_{\y \in \YCal} \ f(\hat{\x}, \y) - \frac{\epsilon}{4}. 
\end{equation*}
On the other hand, 
\begin{equation*}
\min_{\x \in \XCal} \ \left\{f(\x, \hat{\y}) + \frac{\epsilon\|\x - \x_0\|^2}{8\Dx^2} - \frac{\epsilon\|\hat{\y} - \y_0\|^2}{8\Dy^2}\right\} \ \leq \ \min_{\x \in \XCal} \ f(\x, \hat{\y}) + \frac{\epsilon}{4}. 
\end{equation*}
Putting these pieces together yields that $\max_{\y \in \YCal} f(\hat{\x}, \y) - \min_{\x \in \XCal} f(\x, \hat{\y}) \leq \epsilon$. 

Furthermore, letting $\kappax = 4\ell\Dx^2/\epsilon$ and $\kappay = 2\ell\Dy^2/\epsilon$ in the gradient complexity bound presented in Theorem~\ref{thm:scsc}, we conclude that the total number of gradient evaluations is bounded by 
\begin{equation*}
O\left(\frac{\ell\Dx\Dy}{\epsilon}\log^3\left(\frac{\ell D^2}{\epsilon}\right)\right). 
\end{equation*}
This completes the proof. 


\section{Proofs for Nonconvex-Concave Settings}
In this section, we present proofs for all results in Section \ref{sec:result_nonconvex_concave} and Section \ref{sec:app-nonconvex}

\subsection{Proof of Theorem~\ref{thm:nsc}}
Using the definition of $g_t$, we have
\begin{equation*}
\max_{\y \in \YCal} f(\x_{t+1}, \y) + \ell\|\x_{t+1} - \x_t\|^2 \ \leq \ \min_{\x \in \XCal} \left\{\max_{\y \in \YCal} f(\x, \y) + \ell\|\x - \x_t\|^2\right\} + \delta. 
\end{equation*}
This implies that 
\begin{equation*}
\Phi(\x_{t+1}) + \ell\|\x_{t+1} - \x_t\|^2 \ \leq \ \min_{\x \in \XCal} \left\{\Phi(\x) + \ell\|\x - \x_t\|^2\right\} + \delta \ \leq \ \Phi(\x_t) + \delta.  
\end{equation*}
Equivalently, we have
\begin{equation}\label{Theorem:nsc-inequality-first}
\|\x_{t+1} - \x_t\|^2 \ \leq \ \frac{\Phi(\x_t) - \Phi(\x_{t+1}) + \delta}{\ell}. 
\end{equation}
Note that the function $\Phi(\cdot) + \ell\|\cdot - \x_t\|^2$ is $\ell$-strongly convex and its minimizer $\x_t^*$ is well defined and unique~\citep{Davis-2019-Stochastic}. Since the function $\Phi(\cdot) + \ell\|\cdot - \x_t\|^2$ is $\ell$-strongly convex, we derive from~\citet[Theorem~2.1.5]{Nesterov-2018-Lectures} that
\begin{equation}\label{Theorem:nsc-inequality-second}
\|\x_{t+1} - \x_t^*\|^2 \ \leq \ \frac{2}{\ell}\left(\Phi(\x_{t+1}) + \ell\|\x_{t+1} - \x_t\|^2 - \min_{\x \in \br^m} \left\{\Phi(\x) + \ell\|\x - \x_t\|^2\right\}\right) \ \leq \ \frac{2\delta}{\ell}. 
\end{equation}
Since $\Phi$ is differentiable, we have 
\begin{equation*}
\left\|\x_t^* - \PX\left(\x_t^* - \frac{\grad\Phi(\x_t^*) + 2\ell(\x_t^* - \x_t)}{\ell}\right)\right\| \ = \ 0. 
\end{equation*}
Therefore, we have
\begin{equation*}
\left\|\x_{t+1} - \PX\left(\x_{t+1} - \frac{\grad\Phi(\x_{t+1})}{\ell}\right)\right\| \ \leq \ 2\|\x_{t+1} - \x_t^*\| + 2\|\x_t - \x_t^*\| + \frac{\|\nabla\Phi(\x_t^*) - \grad\Phi(\x_{t+1})\|}{\ell}. 
\end{equation*}
Since $\Phi(\cdot)$ is $2\kappay\ell$-smooth, we have $\|\nabla \Phi(\x_{t+1}) - \nabla \Phi(\x_t^*)\| \leq 2\kappay\ell\|\x_{t+1} - \x_t^*\|$. Putting these pieces together yields that 
\begin{eqnarray}\label{Theorem:nsc-inequality-third}
\left\|\x_{t+1} - \PX\left(\x_{t+1} - \frac{\grad\Phi(\x_{t+1})}{\ell}\right)\right\| & \leq & (2\kappay + 2)\|\x_{t+1} - \x_t^*\| + 2\|\x_t - \x_t^*\| \\
& \overset{\kappay \geq 1}{\leq} & 6\kappay\|\x_{t+1} - \x_t^*\| + 2\|\x_{t+1} - \x_t\|. \nonumber
\end{eqnarray}
Putting Eq.~\eqref{Theorem:nsc-inequality-first}, Eq.~\eqref{Theorem:nsc-inequality-second} and Eq.~\eqref{Theorem:nsc-inequality-third} together with the Cauchy-Schwarz inequality yields 
\begin{eqnarray*}
(\ell\left\|\x_{t+1} - \PX\left(\x_{t+1} - (1/\ell)\grad\Phi(\x_{t+1})\right)\right\|)^2 & \leq & 72\kappay^2\ell^2\|\x_{t+1} - \x_t^*\|^2 + 8\ell^2\|\x_{t+1} - \x_t\|^2 \\ 
& \leq & 8\ell\left(\Phi(\x_t) - \Phi(\x_{t+1}) + \delta\right) + 144\kappay^2\ell\delta. 
\end{eqnarray*}
Summing up the above inequality over $t = 0, 1, \ldots, T-1$ and dividing it by $T$ yields that 
\begin{eqnarray*}
\frac{1}{T}\left(\sum_{t=0}^{T-1} (\ell\left\|\x_{t+1} - \PX\left(\x_{t+1} - (1/\ell)\grad\Phi(\x_{t+1})\right)\right\|)^2\right) & \leq & \frac{8\ell(\Phi(\x_0) - \Phi(\x_T))}{T} + 8\ell\delta + 144\kappa_\y^2\ell\delta \\ 
& \overset{\kappay \geq 1}{\leq} & \frac{8\ell(\Phi(\x_0) - \Phi(\x_T))}{T} + 152\kappay^2\ell\delta. 
\end{eqnarray*}
Since $\hat{\x} = \x_s$ is uniformly chosen from $\{\x_s\}_{1 \leq s \leq T}$ and $\delta \leq \epsilon^2/(10\kappay)^4\ell$, we have
\begin{eqnarray*}
\EE\left[(\ell\left\|\hat{\x} - \PX\left(\hat{\x} - (1/\ell)\grad\Phi(\hat{\x})\right)\right\|)^2\right] & = & \frac{1}{T}\left(\sum_{t=0}^{T-1} (\ell\left\|\x_{t+1} - \PX\left(\x_{t+1} - (1/\ell)\grad\Phi(\x_{t+1})\right)\right\|)^2\right) \\ 
& \leq & \frac{8\ell(\Phi(\x_0) - \Phi(\x_T))}{T} + 152\kappay^2\ell\delta \ \leq \ \frac{8\ell\Delta_\Phi}{T} + \frac{\epsilon^2}{8}. 
\end{eqnarray*}
Using the Markov inequality, we conclude that there exists $T > c\ell\Delta_\Phi\epsilon^{-2}$, where the output $\hat{\x}$ will satisfy $\ell\left\|\hat{\x} - \PX\left(\hat{\x} - (1/\ell)\grad\Phi(\hat{\x})\right)\right\| \leq \epsilon/2$ with probability at least $2/3$. 

For simplicity, we denote $\hat{\y}^+ = \PY[\hat{\y} + (1/\ell)\grady f(\hat{\x}, \hat{\y})]$. Since $\hat{\y}$ is obtained by running AGD on $-f(\hat{\x}, \cdot)$ to optimal with tolerance $\delta \leq \epsilon^2/(10\kappay)^4\ell$, and $f(\hat{\x}, \cdot)$ is $\mu_\y$-concave function, we know that $\delta$-optimality guarantees:
\begin{equation*}
\ell\|\hat{\y}^+ - \hat{\y}\| \leq \epsilon, \quad \|\hat{\y}^+ - \y^\star(\hat{\x})\| \leq \frac{\epsilon}{2\ell}. 
\end{equation*}
Putting these pieces together yields that 
\begin{eqnarray*}
\ell\left\|\hat{\x} - \PX\left(\hat{\x} - (1/\ell)\gradx f(\hat{\x}, \hat{\y}^+)\right)\right\| & \leq & \ell\left\|\hat{\x} - \PX\left(\hat{\x} - (1/\ell)\grad\Phi(\hat{\x})\right)\right\| + \|\grad\Phi(\hat{\x}) - \gradx f(\hat{\x}, \hat{\y}^+)\| \\ 
& \leq & \ell\left\|\hat{\x} - \PX\left(\hat{\x} - (1/\ell)\grad\Phi(\hat{\x})\right)\right\| + \ell\|\hat{\y}^+ - \y^\star(\hat{\x})\| \\
& \leq & \epsilon. 
\end{eqnarray*}
This implies that $(\hat{\x}, \hat{\y})$ is an $\epsilon$-stationary point. Furthermore, we call the solver $\textsc{\AGPROX}$ at each iteration. Using Theorem~\ref{Theorem:maximin-AG2} and $\delta \leq \epsilon^2/(10\kappay)^4\ell$, the number of gradient evaluations at each iteration is bounded by  
\begin{equation*}
O\left(\sqrt{\kappay}\log\left(\frac{\kappay^5\ell^2(\tilde{D}_\x^2 + \Dy^2)}{\epsilon^2}\right)\log\left(\frac{\kappay^4\ell^2\Dy^2}{\epsilon^2}\right)\right). 
\end{equation*}
Therefore, we conclude that the total number of gradient evaluations is bounded by 
\begin{equation*}
O\left(\frac{\ell\Delta_\Phi}{\epsilon^2}\cdot\sqrt{\kappay}\log^2\left(\frac{\kappay\ell (\tilde{D}_\x^2 + \Dy^2)}{\epsilon}\right)\right). 
\end{equation*}
This completes the proof.  
\subsection{Proof of Corollary~\ref{cor:nc}}
Recall that the function $\tilde{f}_\epsilon$ is defined by 
\begin{equation*}
\tilde{f}_\epsilon(\x, \y) \ = \ f(\x, \y) - \frac{\epsilon\|\y - \y_0\|^2}{4\Dy}. 
\end{equation*}
This implies that the following statement holds for all $(\x, \y) \in \XCal \times \YCal$ that 
\begin{eqnarray*}
\gradx f(\x, \y) - \gradx\tilde{f}_\epsilon(\x, \y) & = & 0, \\
\|\grady f(\x, \y) - \grady\tilde{f}_\epsilon(\x, \y)\| & \leq & \frac{\epsilon}{2}. 
\end{eqnarray*}
Since $(\hat{\x}, \hat{\y}) = \textsc{\AGNC}(\tilde{f}_\epsilon, \x_0, \y_0, \ell, \epsilon/(2D_{\y}), \epsilon/2, T)$, we have
\begin{equation*}
\ell\|\PX[\hat{\x} - (1/\ell)\gradx \tilde{f}_\epsilon(\hat{\x}, \hat{\y}_\epsilon^+)] - \hat{\x}\| \leq \frac{\epsilon}{2}, \quad \ell\|\hat{\y}_\epsilon^+ - \hat{\y}\| \leq \frac{\epsilon}{2}, \quad \hat{\y}_\epsilon^+ = \PY[\hat{\y} + (1/\ell)\grady \tilde{f}_\epsilon(\hat{\x}, \hat{\y})].  
\end{equation*}
Putting these pieces together with $\hat{\y}^+ = \PY[\hat{\y} + (1/\ell)\grady f(\hat{\x}, \hat{\y})]$ yields that
\begin{equation*}
\ell\|\PX[\hat{\x} - (1/\ell)\gradx f(\hat{\x}, \hat{\y}^+)] - \hat{\x}\| \ \leq \ \frac{\epsilon}{2} + \|\gradx \tilde{f}_\epsilon(\hat{\x}, \hat{\y}_\epsilon^+) - \gradx f(\hat{\x}, \hat{\y}^+)\| \ \leq \ \frac{\epsilon}{2} + \|\grady \tilde{f}_\epsilon(\hat{\x}, \hat{\y}) - \grady f(\hat{\x}, \hat{\y})\| \ \leq \ \epsilon, 
\end{equation*}
and 
\begin{equation*}
\ell\|\hat{\y}^+ - \hat{\y}\| \ \leq \ \ell\|\PY[\hat{\y} + (1/\ell)\grady \tilde{f}_\epsilon(\hat{\x}, \hat{\y})] - \hat{\y}\| + \|\grady f(\x, \y) - \grady\tilde{f}_\epsilon(\x, \y)\| \ \leq \ \epsilon.
\end{equation*}
Therefore, we conclude that $(\hat{\x}, \hat{\y})$ is an $\epsilon$-stationary point of $f$. Furthermore, letting $\kappay = 2\ell\Dy/\epsilon$ in the gradient complexity bound presented in Theorem \ref{thm:nsc}, we conclude that the total number of gradient evaluations is bounded by 
\begin{equation*}
O\left(\frac{\ell\Delta_\Phi}{\epsilon^{2}} \cdot \sqrt{\frac{\ell \Dy}{\epsilon}}\log^2 \left(\frac{\ell(\tilde{D}_\x^2 + \Dy^2)}{\epsilon}\right)\right). 
\end{equation*}
This completes the proof. 

\subsection{Proof of Theorem~\ref{thm:nsc-Moreau}}
Using the same argument as in Theorem~\ref{thm:nsc}, we have 
\begin{equation}\label{Theorem:nsc-inequality-fourth}
\|\x_{t+1} - \x_t\|^2 \ \leq \ \frac{\Phi(\x_t) - \Phi(\x_{t+1}) + \delta}{\ell}. 
\end{equation}
and
\begin{equation}\label{Theorem:nsc-inequality-fifth}
\|\x_{t+1} - \x_t^*\|^2 \ \leq \ \frac{2}{\ell}\left(\Phi(\x_{t+1}) + \ell\|\x_{t+1} - \x_t\|^2 - \min_{\x \in \br^m} \left\{\Phi(\x) + \ell\|\x - \x_t\|^2\right\}\right) \ \leq \ \frac{2\delta}{\ell}. 
\end{equation}
Since $\Phi$ is differentiable, we have $\nabla\Phi(\x_t^*) + 2\ell(\x_t^* - \x_t) = 0$ which implies $\|\nabla \Phi(\x_t^*)\| = 2\ell\|\x_t^* - \x_t\|$. Since $\Phi(\cdot)$ is $2\kappay\ell$-smooth, we have $\|\nabla \Phi(\x_{t+1}) - \nabla \Phi(\x_t^*)\| \leq 2\kappay\ell\|\x_{t+1} - \x_t^*\|$. Putting these pieces together yields that 
 \begin{eqnarray}\label{Theorem:nsc-inequality-sixth}
\|\nabla \Phi(\x_{t+1})\| & \leq & 2\kappay\ell\|\x_{t+1} - \x_t^*\| + 2\ell\|\x_t^* - \x_t\| \ \leq \ (2\kappay\ell + 2\ell)\|\x_{t+1} - \x_t^*\| + 2\ell\|\x_{t+1} - \x_t\| \nonumber \\
& \overset{\kappay \geq 1}{\leq} & 4\kappay\ell\|\x_{t+1} - \x_t^*\| + 2\ell\|\x_{t+1} - \x_t\|.
\end{eqnarray}
Putting Eq.~\eqref{Theorem:nsc-inequality-fourth}, Eq.~\eqref{Theorem:nsc-inequality-fifth} and Eq.~\eqref{Theorem:nsc-inequality-sixth} together with the Cauchy-Schwarz inequality yields 
\begin{equation*}
\|\nabla \Phi(\x_{t+1})\|^2 \ \leq \ 32\kappay^2\ell^2\|\x_{t+1} - \x_t^*\|^2 + 8\ell^2\|\x_{t+1} - \x_t\|^2 \ \leq \ 8\ell\left(\Phi(\x_t) - \Phi(\x_{t+1}) + \delta\right) + 64\kappay^2\ell\delta. 
\end{equation*}
Summing up the above inequality over $t = 0, 1, \ldots, T-1$ and dividing it by $T$ yields that 
\begin{equation*}
\frac{1}{T}\left(\sum_{t=0}^{T-1} \|\nabla \Phi(\x_{t+1})\|^2\right) \ \leq \ \frac{8\ell(\Phi(\x_0) - \Phi(\x_T))}{T} + 8\ell\delta + 64\kappa_\y^2\ell\delta \ \overset{\kappay \geq 1}{\leq} \ \frac{8\ell(\Phi(\x_0) - \Phi(\x_T))}{T} + 72\kappay^2\ell\delta. 
\end{equation*}
Since $\hat{\x} = \x_s$ is uniformly chosen from $\{\x_s\}_{1 \leq s \leq T}$ and $\delta \leq \epsilon^2/144\kappay^2\ell$, we have
\begin{equation*}
\EE\left[\|\nabla \Phi(\hat{\x})\|^2\right] \ = \ \frac{1}{T}\left(\sum_{t=0}^{T-1} \|\nabla \Phi(\x_{t+1})\|^2\right) \ \leq \ \frac{8\ell(\Phi(\x_0) - \Phi(\x_T))}{T} + 72\kappay^2\ell\delta \ \leq \ \frac{8\ell\Delta_\Phi}{T} + \frac{\epsilon^2}{2}. 
\end{equation*}
Using the Markov inequality, we conclude that there exists $T > c\ell\Delta_\Phi\epsilon^{-2}$, where the output $\hat{\x}$ will satisfy $\|\grad \Phi(\hat{\x})\| \leq \epsilon$ with probability at least $2/3$. Furthermore, we call the solver $\textsc{\AGPROX}$ at each iteration. Using Theorem~\ref{Theorem:maximin-AG2} and $\delta \leq \epsilon^2/144\kappay^2\ell$, the number of gradient evaluations at each iteration is bounded by  
\begin{equation*}
O\left(\sqrt{\kappay}\log\left(\frac{\kappay^3\ell^2(\tilde{D}_\x^2 + \Dy^2)}{\epsilon^2}\right)\log\left(\frac{\kappay^2\ell^2\Dy^2}{\epsilon^2}\right)\right). 
\end{equation*}
Therefore, we conclude that the total number of gradient evaluations is bounded by 
\begin{equation*}
O\left(\frac{\ell\Delta_\Phi}{\epsilon^2}\cdot\sqrt{\kappay}\log^2\left(\frac{\kappay\ell (\tilde{D}_\x^2 + \Dy^2)}{\epsilon}\right)\right). 
\end{equation*}
This completes the proof. 

\subsection{Proof of Corollary~\ref{cor:nc-Moreau}}
Recall that the function $\bar{f}_\epsilon$ is defined by
\begin{equation*}
\bar{f}_\epsilon(\x, \y) \ = \ f(\x, \y) - \frac{\epsilon^2\|\y - \y_0\|^2}{200\ell\Dy^2}. 
\end{equation*}
This implies that the following statement holds for all $(\x, \y) \in \XCal \times \YCal$ that 
\begin{eqnarray*}
\gradx f(\x, \y) - \gradx\bar{f}_\epsilon(\x, \y) & = & 0, \\
\|\grady f(\x, \y) - \grady\bar{f}_\epsilon(\x, \y)\| & \leq & \frac{\epsilon^2}{100\ell\Dy}. 
\end{eqnarray*}
Using Theorem~\ref{thm:nsc-Moreau} and letting $\y_\epsilon^\star(\cdot) = \argmin_{\y \in \YCal} \bar{f}_\epsilon(\cdot, \y)$, we have
\begin{eqnarray*}
\|\gradx \bar{f}_\epsilon(\hat{\x}, \y_\epsilon^\star(\hat{\x}))]\| & \leq & \frac{\epsilon}{10}, \\ 
\ell\|\PY[\y_\epsilon^\star(\hat{\x}) + (1/\ell)\grady \bar{f}_\epsilon(\hat{\x}, \y_\epsilon^\star(\hat{\x}))] - \y_\epsilon^\star(\hat{\x})\| & = & 0. 
\end{eqnarray*}
For simplicity, we define $\y_\epsilon^+ = \PY[\y_\epsilon^\star(\hat{\x}) + (1/\ell)\grady f(\hat{\x}, \y_\epsilon^\star(\hat{\x}))]$. Then, we have
\begin{equation*}
\|\gradx f(\hat{\x}, \y_\epsilon^+)\| \leq \frac{\epsilon}{10} + \frac{\epsilon^2}{50\ell\Dy}, \quad \ell\|\y_\epsilon^+ - \y_\epsilon^\star(\hat{\x})\| \leq \frac{\epsilon^2}{50\ell\Dy}. 
\end{equation*}
Now let $\x^\star(\hat{\x}) = \argmin_{\x\in\mathbb{R}^m} \Phi_{1/2\ell}(\x):= \Phi(\x) + \ell \norm{\x - \hat{\x}}^2$, we have
\begin{equation*}
\norm{\grad \Phi_{1/2\ell}(\hat{\x})}^2 = 4\ell^2\norm{\hat{\x} - \x^\star(\hat{\x})}^2
\end{equation*}
Since $\Phi(\cdot) + \ell\|\cdot - \hat{\x}\|^2$ is $\ell/2$-strongly-convex, we have
\begin{eqnarray*}
\lefteqn{\max_{\y \in \YCal} f(\hat{\x}, \y) - \max_{\y \in \YCal} f(\x^*(\hat{\x}), \y) - \ell\|\x^*(\hat{\x}) - \hat{\x}\|^2} \\ 
& = & \Phi(\hat{\x}) - \Phi(\x^*(\hat{\x})) - \ell\|\x^*(\hat{\x}) - \hat{\x}\|^2 \ \geq \ \frac{\ell\|\hat{\x} - \x^*(\hat{\x})\|^2}{4} = \frac{\norm{\grad \Phi_{1/2\ell}(\hat{\x})}^2}{16\ell}. 
\end{eqnarray*}
Furthermore, we have
\begin{eqnarray*}
\lefteqn{\max_{\y \in \YCal} f(\hat{\x}, \y) - \max_{\y \in \YCal} f(\x^*(\hat{\x}), \y) - \ell\|\x^*(\hat{\x}) - \hat{\x}\|^2} \\
& = & \max_{\y \in \YCal} f(\hat{\x}, \y) - f(\hat{\x}, \y_\epsilon^+) + f(\hat{\x}, \y_\epsilon^+) - \max_{\y \in \YCal} f(\x^*(\hat{\x}), \y) - \ell\|\x^*(\hat{\x}) - \hat{\x}\|^2 \\
& \leq & \max_{\y \in \YCal} f(\hat{\x}, \y) - f(\hat{\x}, \y_\epsilon^+) + \left(f(\hat{\x}, \y_\epsilon^+)- f(\x^*(\hat{\x}), \y_\epsilon^+) - \ell\|\x^*(\hat{\x}) - \hat{\x}\|^2\right) \\
& \leq & \max_{\y \in \YCal} f(\hat{\x}, \y) - f(\hat{\x}, \y_\epsilon^+) + \left(\|\hat{\x} - \x^*(\hat{\x})\|\|\gradx f(\hat{\x}, \y_\epsilon^+)\| - \ell\|\hat{\x} - \x^*(\hat{\x})\|^2\right) \\
& \leq & \max_{\y \in \YCal} f(\hat{\x}, \y) - f(\hat{\x}, \y_\epsilon^+) + \frac{\|\gradx f(\hat{\x}, \y_\epsilon^+)\|^2}{4\ell}.
\end{eqnarray*}
Recall that $\y_\epsilon^+ = \PY[\y_\epsilon^\star(\hat{\x}) + (1/\ell)\grady f(\hat{\x}, \y_\epsilon^\star(\hat{\x}))]$, we have
\begin{equation*}
(\y - \y_\epsilon^+)^\top\left(\y_\epsilon^+ - \y_\epsilon^\star(\hat{\x}) - (1/\ell)\grady f(\hat{\x}, \y_\epsilon^\star(\hat{\x}))\right) \geq 0 \textnormal{ for all } \y \in \YCal. 
\end{equation*}
Together with the $\ell$-smoothness of the function $f(\hat{\x}, \cdot)$ and the boundedness of $\YCal$, we have
\begin{equation*}
f(\hat{\x}, \y) - f(\hat{\x}, \y_\epsilon^+) \leq \frac{\ell}{2}(\|\y - \y_\epsilon^\star(\hat{\x})\|^2 - \|\y - \y_\epsilon^+\|^2) \leq \ell D_\y \|\y_\epsilon^+ - \y_\epsilon^\star(\hat{\x})\| \textnormal{ for all } \y \in \YCal. 
\end{equation*}
Putting these pieces together yields that 
\begin{equation*}
\max_{\y \in \YCal} f(\hat{\x}, \y) - \max_{\y \in \YCal} f(\x^\star(\hat{\x}), \y) - \ell\|\x^\star(\hat{\x}) - \hat{\x}\|^2 \leq \ell D_\y\|\y_\epsilon^+ - \y_\epsilon^\star(\hat{\x})\| + \frac{\|\gradx f(\hat{\x}, \y_\epsilon^+)\|^2}{4\ell}. 
\end{equation*}
Since a point $(\hat{\x}, \y_\epsilon^\star(\hat{\x}))$ satisfies that  
\begin{equation*}
\|\gradx f(\hat{\x}, \y_\epsilon^+)\| \leq \frac{\epsilon}{10} + \frac{\epsilon^2}{50\ell\Dy}, \quad \ell\|\y_\epsilon^+ - \y_\epsilon^\star(\hat{\x})\| \leq \frac{\epsilon^2}{50\ell\Dy},  
\end{equation*}
we have (assume that $\epsilon \precsim \ell D_\y$ without loss of generality)
\begin{equation*}
\max_{\y \in \YCal} f(\hat{\x}, \y) - \max_{\y \in \YCal} f(\x^\star(\hat{\x}), \y) - \ell\|\x^\star(\hat{\x}) - \hat{\x}\|^2 \leq \frac{\epsilon^2}{50\ell} + \frac{\|\gradx f(\hat{\x}, \y_\epsilon^+)\|^2}{4\ell}. 
\end{equation*}
Putting these pieces together yields that $\|\grad \Phi_{1/2\ell}(\hat{\x})\| \le \epsilon$. Furthermore, letting $\kappay = 100\ell^2\Dy^2/\epsilon^2$ in the gradient complexity bound presented in Theorem \ref{thm:nsc-Moreau}, we conclude that the total number of gradient evaluations is bounded by 
\begin{equation*}
O\left(\frac{\ell^2\Dy\Delta_\Phi}{\epsilon^3} \log^2 \left(\frac{\ell(\tilde{D}_\x^2 + \Dy^2)}{\epsilon}\right)\right). 
\end{equation*}
This completes the proof. 


\section{Proof of Technical Lemmas}\label{app:lemmas}
In this section, we provide complete proofs for the lemmas in the paper. 
\subsection{Proof of Lemma~\ref{Lemma:NC-Moreau-Envelope}}
We provide a proof for an expanded version of Lemma~\ref{Lemma:NC-Moreau-Envelope}. 
\begin{lemma}
If $\Phi$ is $\ell$-weakly convex, we have
\begin{enumerate}[(a)]
\item $\Phi_{1/2\ell}(\x)$ and $\prox_{\Phi/2\ell}(\x) = \argmin \Phi(\w) + \ell\|\w - \x\|^2$ are well defined for any $\x \in \br^m$. 
\item $\Phi(\prox_{\Phi/2\ell}(\x)) \leq \Phi(\x)$ for any $\x \in \br^m$. 
\item $\Phi_{1/2\ell}$ is $4\ell$-smooth with $\grad\Phi_{1/2\ell}(\x) = 2\ell(\x - \prox_{\Phi/2\ell}(\x))$. 
\end{enumerate}
\end{lemma}
\begin{proof}
Since $\Phi$ is $\ell$-weakly convex, $\Phi(\cdot) + (\ell/2)\left\|\cdot - \x\right\|^2$ is convex for any $\x \in \br^m$. This implies that $\Phi(\cdot) + \ell\left\|\cdot - \x\right\|^2$ is $(\ell/2)$-strongly convex and $\Phi_{1/2\ell}(\x)$ and $\prox_{\Phi/2\ell}(\x)$ are well defined. For any $\x \in \br^m$, the definition of $\prox_{\Phi/2\ell}(\x)$ implies that 
\begin{equation*}
\Phi(\prox_{\Phi/2\ell}(\x)) \ \leq \ \Phi_{1/2\ell}(\prox_{\Phi/2\ell}(\x)) \ \leq \ \Phi(\x). 
\end{equation*}
By~\citet[Lemma~2.2]{Davis-2019-Stochastic}, $\Phi_{1/2\ell}$ is differentiable with $\grad\Phi_{1/2\ell}(\x) = 2\ell(\x - \prox_{\Phi/2\ell}(\x))$. Since $\prox_{\Phi/2\ell}$ is $1$-Lipschitz, we $\|\grad\Phi_{1/2\ell}(\x) - \grad\Phi_{1/2\ell}(\x')\| \leq 4\ell\|\x - \x'\|$. Therefore, the function $\Phi_{1/2\ell}$ is $4\ell$-smooth.
\end{proof}

\subsection{Proof of Lemma~\ref{Lemma:NC-Stationary}}
Denote $\hat{\x} := \prox_{\Phi/2\ell}(\x)$, part (c) in Lemma~\ref{Lemma:NC-Moreau-Envelope} implies
\begin{equation*}
\|\hat{\x} - \x\| \ = \ \frac{\|\grad\Phi_{1/2\ell}(\x)\|}{2\ell}. 
\end{equation*}
Furthermore, we have $2\ell(\x - \hat{\x}) \in \partial \Phi(\hat{\x})$. Putting these pieces together yields the desired result.

\subsection{Proof of Lemma~\ref{Lemma:minimaxAG2-structure}}
\paragraph{Part (a):} Let $\x, \x' \in \br^m$, the points $\y_g^*(\x)$ and $\y_g^*(\x')$ satisfy that 
\begin{align}
(\y - \y_g^*(\x))^\top \grady g(\x, \y_g^*(\x)) & \ \leq \ 0, \qquad \forall \y \in \YCal, \label{inequality-structure-first} \\
(\y-\y_g^*(\x'))^\top \grady g(\x', \y_g^*(\x')) & \ \leq \ 0, \qquad \forall \y \in \YCal. \label{inequality-structure-second}
\end{align}
Summing up Eq.~\eqref{inequality-structure-first} with $\y = \y_g^*(\x')$ and Eq.~\eqref{inequality-structure-second} with $\y = \y_g^*(\x)$ yields
\begin{equation*}
(\y_g^*(\x') - \y_g^*(\x))^\top(\grady g(\x, \y_g^*(\x)) - \grady g(\x', \y_g^*(\x'))) \ \leq \ 0. 
\end{equation*} 
Since $g(\x, \cdot)$ is $\mu_\y$-strongly concave, we have
\begin{equation*}
(\y_g^*(\x') - \y_g^*(\x))^\top(\grady g(\x, \y_g^*(\x')) - \grady g(\x, \y_g^*(\x))) + \mu_\y\|\y_g^*(\x') - \y_g^*(\x)\|^2 \ \leq \ 0. 
\end{equation*}
Summing up the above two inequalities yields that 
\begin{equation*}
(\y_g^*(\x') - \y_g^*(\x))^\top(\grady g(\x, \y_g^*(\x')) - \grady g(\x', \y_g^*(\x'))) + \mu_\y\|\y_g^*(\x') - \y_g^*(\x)\|^2 \ \leq \ 0. 
\end{equation*}
Since $\grady g$ is $\ell$-Lipschitz, we have
\begin{equation*}
\mu_\y\|\y_g^*(\x') - \y_g^*(\x)\|^2 \ \leq \ \ell\|\y_g^*(\x') - \y_g^*(\x)\|\|\x' - \x\|. 
\end{equation*}
Therefore, we conclude that the function $\y_g^*(\cdot)$ is $\kappay$-Lipschitz. 
\paragraph{Part (b):} Since the function $\y_g^*(\cdot)$ is unique,  Danskin's theorem~\citep{Rockafellar-1970-Convex} implies that $\Phi_g$ is differentiable and $\grad \Phi_g(\cdot) = \gradx g(\cdot, \y_g^*(\cdot))$. Let $\x, \x' \in \br^m$, we have
\begin{eqnarray*}
\|\grad \Phi_g(\x) - \grad \Phi_g(\x')\| & = & \|\gradx g(\x, \y_g^*(\x)) - \gradx g(\x', \y_g^*(\x'))\| \ \leq \ \ell\|\x - \x'\| + \ell\|\y_g^*(\x) - \y_g^*(\x')\| \\
& \overset{\bar{\kappa} \geq 1}{\leq} & \kappay\ell\|\x - \x'\| + \ell\|\y_g^*(\x) - \y_g^*(\x')\|. 
\end{eqnarray*}
Since $\y_g^*(\cdot)$ is $\kappay$-Lipschitz, the function $\Phi_g$ is $2\kappay\ell$-smooth. Furthermore, let $\x, \x' \in \br^m$, we have
\begin{eqnarray*}
\Phi_g(\x') - \Phi_g(\x) - (\x' - \x)^\top\grad\Phi_g(\x) & = & g(\x', \y_g^*(\x')) - g(\x, \y_g^*(\x)) - (\x' - \x)^\top\gradx g(\x, \y_g^*(\x)) \\
& \geq & g(\x', \y_g^*(\x)) - g(\x, \y_g^*(\x)) - (\x' - \x)^\top\gradx g(\x, \y_g^*(\x)). 
\end{eqnarray*}
Since $g(\cdot, \y)$ is $\mu_\x$-strongly convex for each $\y \in \YCal$, we have
\begin{equation*}
\Phi_g(\x') - \Phi_g(\x) - (\x' - \x)^\top\grad\Phi_g(\x) \ \geq \ \frac{\mu_\x\|\x' - \x\|^2}{2}.  
\end{equation*}
Therefore, the function $\Phi_g$ is $\mu_\x$-strongly convex. 
\paragraph{Part (c):} Let $\y, \y' \in \br^n$, the points $\x_g^*(\y)$ and $\x_g^*(\y')$ satisfy that  
\begin{align}
(\x - \x_g^*(\y))^\top \gradx g(\x_g^*(\y), \y) & \ \geq \ 0, \qquad \forall \x \in \XCal, \label{inequality-structure-third} \\
(\x - \x_g^*(\y'))^\top \gradx g(\x_g^*(\y'), \y') & \ \geq \ 0, \qquad \forall \x \in \XCal. \label{inequality-structure-fourth}
\end{align}
Summing up Eq.~\eqref{inequality-structure-third} with $\x = \x_g^*(\y')$ and Eq.~\eqref{inequality-structure-fourth} with $\x = \x_g^*(\y)$ yields 
\begin{equation*}
(\x_g^*(\y') - \x_g^*(\y))^\top(\gradx g(\x_g^*(\y), \y) - \gradx g(\x_g^*(\y'), \y')) \ \geq \ 0. 
\end{equation*} 
Since $g(\cdot, \y)$ is $\mu_\x$-strongly convex, we have
\begin{equation*}
(\x_g^*(\y') - \x_g^*(\y))^\top(\gradx g(\x_g^*(\y'), \y') - \gradx g(\x_g^*(\y), \y')) - \mu_\x\|\x_g^*(\y') - \x_g^*(\y)\|^2 \ \geq \ 0. 
\end{equation*}
Summing up the above two inequalities yields that 
\begin{equation*}
(\x_g^*(\y') - \x_g^*(\y))^\top(\gradx g(\x_g^*(\y), \y) - \gradx g(\x_g^*(\y), \y')) - \mu_\x\|\x_g^*(\y') - \x_g^*(\y)\|^2 \ \geq \ 0. 
\end{equation*}
Since $\gradx g$ is $\ell$-smooth, we have
\begin{equation*}
\mu_\x\|\x_g^*(\y') - \x_g^*(\y)\|^2 \ \leq \ \ell\|\x_g^*(\y') - \x_g^*(\y)\|\|\y' - \y\|. 
\end{equation*}
Therefore, we conclude that the function $\x_g^*$ is $\kappa_\x$-Lipschitz.
\paragraph{Part (d):} Since the function $\x_g^*(\cdot)$ is unique, Danskin's theorem~\citep{Rockafellar-1970-Convex} implies that $\Psi_g$ is differentiable and $\grad \Psi_g(\cdot) = \grady g(\x_g^*(\cdot), \cdot)$. Let $\y, \y' \in \br^n$, we have
\begin{equation*}
\|\grad \Psi_g(\y) - \grad \Psi_g(\y')\| \ = \ \|\grady g(\x_g^*(\y), \y) - \grady g(\x_g^*(\y'), \y)\| \ \leq \ \ell\|\x_g^*(\y) - \x_g^*(\y')\| + \ell\|\y - \y'\|. 
\end{equation*}
Since $\x_g^*(\cdot)$ is $\kappax$-Lipschitz, the function $\Psi_g$ is $2\kappax\ell$-smooth. Furthermore, let $\y, \y' \in \br^n$, we have
\begin{eqnarray*}
\Psi_g(\y) - \Psi_g(\y') - (\y - \y')^\top\grad\Psi_g(\y) & = & g(\x_g^*(\y), \y) - g(\x_g^*(\y'), \y') - (\y - \y')^\top\grady g(\x_g^*(\y), \y) \\
& \geq & g(\x_g^*(\y), \y) - g(\x_g^*(\y), \y') - (\y - \y')^\top\grady g(\x_g^*(\y), \y). 
\end{eqnarray*}
Since $g(\x, \cdot)$ is $\mu_\y$-strongly concave for each $\x \in \XCal$, we have
\begin{equation*}
\Psi_g(\y) - \Psi_g(\y') - (\y - \y')^\top\grad\Psi_g(\y) \ \geq \ \frac{\mu_\y\|\y' - \y\|^2}{2}.  
\end{equation*}
Therefore, the function $\Psi_g$ is $\mu_\y$-strongly concave. 

\end{document}